\documentclass[11pt]{amsart}

\usepackage{amsfonts,epsfig}
\usepackage{latexsym}
\usepackage{amssymb}
\usepackage{amsmath}
\usepackage{amsthm}
\usepackage{graphics}
\usepackage[all]{xy}
\usepackage[T2A]{fontenc}
\usepackage{multirow}
\usepackage{array, booktabs, ctable}
\usepackage{hyperref}
\usepackage{color}
\usepackage{mathtools}
\usepackage{enumerate}
\usepackage{url}

\newtheorem{defn}{Definition}[section]

\newcommand{\Mod}[1]{\ (\mathrm{mod}\ #1)}

\usepackage[OT2,T1]{fontenc}
\DeclareSymbolFont{cyrletters}{OT2}{wncyr}{m}{n}
\DeclareMathSymbol{\Sha}{\mathalpha}{cyrletters}{"58}

\newtheorem{corollary}[defn]{Corollary}
\newtheorem{lemma}[defn]{Lemma}

\newtheorem{theorem}[defn]{Theorem}

\newtheorem{proposition}[defn]{Proposition}

\theoremstyle{definition}
\newtheorem*{remark}{Remark}

\newtheorem{conjecture}[defn]{Conjecture}

\newcommand{\lmfdbecx}[3]{\href{https://www.lmfdb.org/EllipticCurve/Q/#1/#2/#3}{#1.#2#3}}

\newcommand{\Q}{\mathbb Q}
\newcommand{\Z}{\mathbb Z}

\newcommand{\Gal}{\operatorname{Gal}}

\newcommand{\ord}{\operatorname{ord}}
\newcommand{\GL}{\operatorname{GL}}
\newcommand{\PGL}{\operatorname{PGL}}

\newcommand{\rk}{\operatorname{rk}}
\newcommand{\Rk}{\operatorname{rk^{an}}}
\newcommand{\Kum}{\mathbb{Q}^\times/(\mathbb{Q}^\times)^2}

\newcommand{\coker}{\operatorname{coker}}

\begin{document}

\bibliographystyle{alphaurl}

\title[Torsion]{Existence and non-existence of rational elliptic curves with prescribed torsion subgroups over quadratic fields}

\author{\"{O}mer Avci}

\address{Dept. of Mathematics and Statistics, University of Ottawa}
\email{oavci@uotttawa.ca} 



\begin{abstract} 
Let $K=\mathbb{Q}(\sqrt{-p})$ be a quadratic field for an odd prime $p$. We show that
there exist infinitely many primes $p$ for which no elliptic curve $E/\mathbb{Q}$
has torsion subgroup $\mathbb{Z}/2\mathbb{Z}\times \mathbb{Z}/2N\mathbb{Z}$ over $K$
for $N=5,6$. We also prove that there exist infinitely many primes $p$ for which
there are infinitely many elliptic curves $E/\mathbb{Q}$ with this torsion
structure, conditional on the parity conjecture. Using these results, we obtain
new torsion classification results over Kummer extensions of cyclotomic fields
and over composites of $\mathbb{Z}_p$-extensions of number fields, refining and
extending our previous work.
\end{abstract}

\maketitle

\section{Introduction and Notation}

Let $K$ denote a number field, and let $E$ be an elliptic curve over $K$. The Mordell-Weil theorem states that the set of $K$-rational points on $E$ forms a finitely generated abelian group. More precisely, denoting $E(K)$ as the set of $K$-rational points on $E$, we have
\begin{equation*}
    E(K) \cong E(K)_{\text{tors}} \oplus \mathbb{Z}^{\rk(E(K))}.
\end{equation*}
Here, $\rk(E(K))$ is a nonnegative integer called the rank, and $E(K)_{\text{tors}}$ is a finite group known as the torsion subgroup of $E$ over $K$.

Classification of torsion subgroups $E(K)_{\mathrm{tors}}$ is an ongoing problem. 
Mazur's celebrated result solved the problem for $K = \Q$, and many results are 
known for other fields. There are several variants of the problem, for instance, 
whether one considers elliptic curves defined over $K$ itself or over $\Q$.  

Chou classified the torsion subgroups $E(\Q^{\mathrm{ab}})_{\mathrm{tors}}$
for elliptic curves $E/\Q$ \cite{Chou}.
We studied the torsion subgroups
$E(\Q(\zeta_p))_{\mathrm{tors}}$ for all primes $p$, obtaining stronger results
for certain families of primes \cite{Omer}. At the time, our sharpest classification result
was the following theorem.

\begin{theorem}[A., Theorem 1.1, \cite{Omer}]
     Let $E/\mathbb{Q}$ be an elliptic curve, and let $p > 3$ be a prime such that $p - 1$ is not divisible by $3$, $4$, or $5$. Then $E(\mathbb{Q}(\zeta_p))_{\text{tors}}$ is either one of the groups from Mazur’s theorem, or one of the following groups:
  \begin{equation*}
    \mathbb{Z}/2\mathbb{Z} \times \mathbb{Z}/10\mathbb{Z}, \quad 
    \mathbb{Z}/2\mathbb{Z} \times \mathbb{Z}/12\mathbb{Z}, \quad 
    \text{or} \quad \mathbb{Z}/16\mathbb{Z}.
 \end{equation*}
 Moreover,  $E(\mathbb{Q}(\zeta_p))_{\text{tors}} = E(\mathbb{Q}(\sqrt{-p}))_{\text{tors}}$ holds.
\end{theorem}

In the same paper, we provided explicit examples realizing
$\Z/2\Z \times \Z/10\Z$ for $p = 59$ and $\Z/2\Z \times \Z/12\Z$ for $p = 47$.
For the case $\Z/16\Z$, we could not find any examples, and were unable to 
show that it is not realized for any prime $p$. In June 2025, Derickx \cite[Theorem 5.2]{derickx16} proved
that 
\begin{equation*}
    E(\Q(\sqrt{-p}))_{\text{tors}}\cong \Z/16\Z
\end{equation*}
 is not possible for any prime $p$ and elliptic curve $E/\Q$. This shows that, excluding $\Z/16\Z$, the statement of our theorem is best possible. 

 The question of determining for which primes $p$
 the groups $\Z/2\Z \times \Z/2N\Z$
 are realized for $N=5,6$
 remains open. In this paper, we prove that for certain families of primes there are no rational elliptic curves with the desired torsion, and for certain families of primes we conjecture, and prove conditionally on the finiteness of the $2$-primary part of the Tate–Shafarevich group of rational elliptic curves, that there exist infinitely many rational elliptic curves with the desired torsion. Our main results are as follows.
\begin{theorem}
      Let $p\equiv 3,7\pmod{20}$ be a prime. Then there does not exist an elliptic curve $E/\Q$ such that
      \begin{equation*}
          E(\mathbb{Q}(\sqrt{-p}))_{\mathrm{tors}} \cong \Z/2\Z \times \Z/10\Z
      \end{equation*}
      Moreover, $E(\Q(\zeta_p))_{\text{tors}}$ is isomorphic to one of the groups in Mazur's theorem or $\Z/2\Z \times \Z/12\Z$.
\end{theorem}

\begin{theorem}
     Let $p\equiv 11\pmod{24}$ be a prime. Then there does not exist an elliptic curve $E/\Q$ such that
     \begin{equation*}
         E(\mathbb{Q}(\sqrt{-p}))_{\mathrm{tors}} \cong \Z/2\Z \times \Z/12\Z.
     \end{equation*}
      Moreover, $E(\Q(\zeta_p))_{\text{tors}}$ is isomorphic to one of the groups in Mazur's theorem or $\Z/2\Z \times \Z/10\Z$.
\end{theorem}

\begin{corollary}
       Let $p\equiv 83,107\pmod{120}$ be a prime. Then there does not exist an elliptic curve $E/\Q$ such that
       \begin{equation*}
           E(\mathbb{Q}(\sqrt{-p}))_{\mathrm{tors}} \cong \Z/2\Z \times \Z/N\Z
       \end{equation*}
       for $N=5$ or $N=6$. Moreover, $E(\Q(\zeta_p))_{\text{tors}}$ is isomorphic to one of the groups in Mazur's theorem.
\end{corollary}

From this corollary, we obtain an infinite family of primes for which the
possible groups appearing as $E(\Q(\zeta_p))_{\mathrm{tors}}$ are exactly the $15$
groups listed in Mazur's theorem. Moreover, thanks to the work of Guzvić and Krijan,
this also allows us to classify $E(\Q(\mu_{p^\infty}))_{\mathrm{tors}}$, since they
proved that for all primes $p>3$,
\begin{equation*}
E(\Q(\mu_{p^\infty}))_{\mathrm{tors}} = E(\Q(\zeta_p))_{\mathrm{tors}}.
\end{equation*}

We also leave the following questions open:

\begin{conjecture}\label{con10}
     Let $p\equiv 11,19\pmod{20}$ be a prime. Then, there exists infinitely many elliptic curves $E/\Q$, satisfying 
     \begin{equation*}
         E(\mathbb{Q}(\sqrt{-p}))_{\mathrm{tors}} \cong \Z/2\Z \times \Z/10\Z.
     \end{equation*}
\end{conjecture}

\begin{conjecture}\label{con12}
      Let $p\equiv -1\pmod{24}$ be a prime. Then, there exists infinitely many elliptic curves $E/\Q$, satisfying 
      \begin{equation*}
          E(\mathbb{Q}(\sqrt{-p}))_{\mathrm{tors}} \cong \Z/2\Z \times \Z/12\Z.
      \end{equation*}
\end{conjecture}

The classification of torsion subgroups $E(K)_{\text{tors}}$ for elliptic curves 
defined over quadratic fields $K$ was completed by Kamienny \cite{kamienny} 
and by Kenku and Momose \cite{kenkumomose}. In addition, the classification 
of possible torsion subgroups $E(K)_{\text{tors}}$ arising from elliptic curves 
$E/\mathbb{Q}$ was carried out by Najman \cite{NajmanQuadratic}. In these 
classifications, one encounters torsion groups that do not appear in Mazur's 
theorem. However, although it is known that such groups are realized over some
quadratic fields, a complete description of which quadratic fields admit which 
torsion structures is still unknown.

This problem has attracted the attention of several authors in recent years. Notably, Banwait and Derickx studied related questions \cite{bander}, and we investigated further aspects \cite{Omer,Omer2}. As another example, Ejder \cite{ejder} generalized Fujita’s \cite{Fujita} classification of torsion subgroups over the maximal quadratic extension of $\mathbb{Q}$ to the cases of $\mathbb{Q}(\sqrt{-1})$ and $\mathbb{Q}(\sqrt{-3})$, relying on the fact that the classification of torsion over cyclotomic quadratic fields had previously been completed by Najman \cite{Najmancyclo}. Extending such classifications to more general quadratic fields may therefore lead to further generalizations of Ejder’s results.

In Section \ref{descentsection}, we give a brief review of the $2$-descent 
method and of results related to quadratic points on the modular curves 
$X_1(2,10)$ and $X_1(2,12)$. These two modular curves have genus $1$, and hence 
are elliptic curves themselves. We study the ranks of their quadratic twists 
and identify infinite families of quadratic twists for which the rank is zero, 
which implies that there are no rational elliptic curves with the desired 
torsion over the corresponding quadratic fields. Finally, we explain how to 
explicitly construct rational elliptic curves with the desired torsion over
quadratic fields for which the rank is positive.

In Section \ref{conditionalsection}, we discuss a result of Monsky \cite{Monsky},
which proves the parity conjecture for rational elliptic curves, conditional on
the finiteness of the $2$-primary part of the Tate-Shafarevich group. We then 
prove our conjectures conditionally on the parity conjecture, or under the
assumption that the $2$-primary part of $\Sha$ is finite, an assumption which
implies the parity conjecture for rational elliptic curves.

In Section \ref{eliminationsection}, we refine and extend several of the
classification results obtained earlier in \cite{Omer}.

Finally, in Section \ref{torsionsection}, we establish new classification 
results related to the problems studied earlier \cite{Omer,Omer2}, using 
the results proved in Section \ref{descentsection} and the methods developed 
and refined in Section \ref{eliminationsection}, together with recent results 
from the literature.

Throughout this paper, elliptic curves are identified using their LMFDB
labels.

In several parts of our analysis, we rely on computations performed with 
\texttt{SageMath}. These codes are publicly available on GitHub at \\
\url{https://github.com/omeravci372742/torsion-subgroups-sagemath-2}.

\noindent \textbf{Acknowledgements.} I would like to thank my advisors, Antonio 
Lei and Payman Eskandari, and my partner in mathematics and in life, Sueda Senturk 
Avci, for their unwavering support, valuable comments, and insightful discussions.

\section{Finding Quadratic Points of two Modular Curves} \label{descentsection}

In this section, we study the problem of finding quadratic points on the 
modular curves $X_1(2,10)$ and $X_1(2,12)$. These two modular curves have genus 
$1$ and hence are elliptic curves themselves. We begin with an important 
theorem stating that all non-cuspidal quadratic points, more precisely, the 
points of $X_1(2,10)(K)$ and $X_1(2,12)(K)$ for a quadratic field $K$ that are 
not cusps, are non-torsion. This theorem also involves the modular curve 
$X_1(11)$, which likewise has genus $1$ and is an elliptic curve.
Since we have already established that $X_1(2,10)$, $X_1(2,12)$, and $X_1(11)$ are 
elliptic curves, we denote their quadratic twists by $d$ as $X_1(2,10)_d$, 
$X_1(2,12)_d$, and $X_1(11)_d$, respectively, for any nonzero integer $d$.

Using this result, we reduce the problem to 
computing the ranks of the quadratic twists of these modular curves. We then 
apply the method of $2$-descent to carry out this computation for certain 
families of quadratic twists. In our calculations, we follow the notes of 
\cite{2-descentpaper}, and we are grateful to Martin Bright for his clearly 
written exposition.

\begin{theorem}[Najman, Kamienny, Theorem 2-8-9-15, \cite{najmankamienny}]\label{noncuspidal}
    If $X_1(11)$, $X_1(2, 10)$ or $X_1(2, 12)$ possess a non-cuspidal quadratic point, then
that point has infinite order.
\end{theorem}
\begin{corollary}
Let us denote the modular curve $X_1(11)$ by the elliptic curve
\begin{equation*}
    X_1(11) : y^2 - y = x^3 - x^2 .
\end{equation*}
Let $d$ be an integer. If $\rk(X_1(11)(\mathbb{Q}(\sqrt{d}))) = 0$, then there is no elliptic curve
$E/\mathbb{Q}(\sqrt{d})$ satisfying
\begin{equation*}
    E(\mathbb{Q}(\sqrt{d}))_{\mathrm{tors}} \cong \mathbb{Z}/11\mathbb{Z}.
\end{equation*}
If $\rk(X_1(11)(\mathbb{Q}(\sqrt{d}))) \geq 1$, then there exist infinitely many elliptic curves
$E/\mathbb{Q}(\sqrt{d})$ satisfying
\begin{equation*}
    E(\mathbb{Q}(\sqrt{d}))_{\mathrm{tors}} \cong \mathbb{Z}/11\mathbb{Z}.
\end{equation*}

The same statement holds for the torsion structures
$\mathbb{Z}/2\mathbb{Z} \times \mathbb{Z}/10\mathbb{Z}$ and
$\mathbb{Z}/2\mathbb{Z} \times \mathbb{Z}/12\mathbb{Z}$ if we replace
$X_1(11)$ by the elliptic curves associated to the modular curves
$X_1(2,10)$ and $X_1(2,12)$, respectively, given by
\begin{equation*}
   X_1(2,10): y^2 = x^3 + x^2 - x,
    \qquad
    X_1(2,12) : y^2 = x^3 - x^2 + x.
\end{equation*}
\end{corollary}

\begin{proof}
   We know that the non-cuspidal $\Q(\sqrt{d})$-points of $X_1(11)$ correspond to elliptic curves $E/\Q(\sqrt{d})$ satisfying $E(\Q(\sqrt{d}))_{\mathrm{tors}} \cong \Z/11\Z$.  
Since $X_1(11)$ is an elliptic curve, we have
\begin{equation*}
  X_1(11)(\Q(\sqrt{d})) \cong X_1(11)(\Q(\sqrt{d}))_{\mathrm{tors}} \oplus \Z^{\rk(X_1(11)(\Q(\sqrt{d})))}
\end{equation*}
by the Mordell-Weil theorem. By Theorem \ref{noncuspidal}, $X_1(11)$ has no non-cuspidal quadratic torsion points. Therefore, it has a non-cuspidal quadratic point if and only if its rank $\rk(X_1(11)(\Q(\sqrt{d})))$ is positive; in this case, it has infinitely many points.  

By using the equality
\begin{equation*}
    \rk(X_1(11)(\Q(\sqrt{d}))) = \rk(X_1(11)(\Q)) + \rk(X_1(11)_d(\Q)),
\end{equation*}
we can reduce the problem to determining whether $\rk(X_1(11)_d(\Q))$ is positive or zero, because it can be readily verified from LMFDB that $\rk(X_1(11)(\Q)) = 0$, which is also evident from Mazur's theorem.  

A similar argument applies to $X_1(2,10)$ and $X_1(2,12)$, since 
\begin{equation*}
    \rk(X_1(2,10)(\Q)) =\rk(X_1(2,12)(\Q)) = 0
\end{equation*}
according to LMFDB.
\end{proof}

In the case $\rk(X_1(2,10)_d)(\mathbb{Q}) = 0$ (respectively for $X_1(2,12)$), there do not exist elliptic curves defined over $\mathbb{Q}(\sqrt{d})$ with the desired torsion. This also implies that no rational elliptic curves with this torsion.  

If, on the other hand, $\rk(X_1(2,10)_d)(\mathbb{Q}) \geq 1$ (respectively for $X_1(2,12)$), there exist infinitely many elliptic curves defined over $\mathbb{Q}(\sqrt{d})$ with the desired torsion. The following theorems demonstrate that, under these conditions, there also exist infinitely many rational elliptic curves with the same torsion structure.

\begin{proposition}\label{infinitely10}
    Let $d$ be an integer, and let $X_1(2,10)_d$ denote the $d$-quadratic twist of the modular curve $X_1(2,10)$. If $\rk(X_1(2,10)_d(\Q)) \geq 1$, then there exist infinitely many elliptic curves $E/\Q$ with 
    \begin{equation*}
        E(\Q(\sqrt{d}))_{\mathrm{tors}} \cong \Z/2\Z \times \Z/10\Z.
    \end{equation*}
\end{proposition}

\begin{proof}
    Since $\rk(E_d(\Q)) \geq 1$, there exist infinitely many rational solutions to the equation 
    \begin{equation*}
        dy^2 = x^3 + x^2 - x.
    \end{equation*}
    In \cite[Theorem 3.2]{JeonKimLee}, it is shown that if $d = d(t) = 8t^3 - 8t^2 + 1$ with $t \in \Q$, and if $E$ is the elliptic curve defined by
    \begin{equation*}
        E: y^2 + (1 - c)xy - by = x^3 - bx^2,
    \end{equation*}
    where 
\begin{equation*}
      b = \frac{t^3(2t^2 - 3t + 1)}{(t^2 - 3t + 1)^2}, \quad 
        c = -\frac{t(2t^2 - 3t + 1)}{t^2 - 3t + 1}, \quad t \neq 0, \frac{1}{2}, 1,
\end{equation*}
    then the torsion subgroup of $E$ over $\Q(\sqrt{d})$ is $\Z/2\Z \times \Z/10\Z$.  

    A change of variables $x = 2t - 1$ shows that a rational solution $(x,y)$ 
    to the first equation determines a rational value of $t$. This rational $t$ 
    in turn produces rational values of $b$ and $c$, so that the corresponding 
    elliptic curve $E$ is defined over $\Q$. Moreover, $E$ has the desired 
    torsion subgroup over $\Q(\sqrt{dy^2})$, which coincides with 
    $\Q(\sqrt{d})$.

    Since the rank is positive, there are infinitely many distinct rational 
    solutions $(x,y)$ to the elliptic curve equation. Distinct $x$-values 
    correspond to distinct $t$-values, which in turn give distinct 
    $j$-invariants. Therefore, we obtain infinitely many distinct rational 
    elliptic curves. This follows because each solution corresponds to a 
    distinct quadratic point on the modular curve $X_1(2,10)$, and each such 
    quadratic point corresponds to a different elliptic curve with the desired 
    torsion.
\end{proof}

\begin{proposition}\label{infinitely12}
     Let $d$ be an integer, and let $X_1(2,12)_d$ denote the $d$-quadratic twist of the modular curve $X_1(2,12)$. If $\rk(X_1(2,12)_d(\Q)) \geq 1$, then there exist infinitely many elliptic curves $E/\Q$ with
     \begin{equation*}
         E(\Q(\sqrt{d}))_{\mathrm{tors}} \cong \Z/2\Z \times \Z/12\Z.
     \end{equation*}
\end{proposition}

\begin{proof}
   Similar to the proof of Proposition \ref{infinitely10}, \cite[Theorem 3.3]{JeonKimLee} states that if $d = d(t) = \frac{t^2 - 1}{t^2 + 3}$ with $t \in \Q$, and if $E$ is the elliptic curve defined by
   \begin{equation*}
       E: y^2 + (1 - c)xy - (c + c^2)y = x^3 - (c + c^2)x^2,
   \end{equation*}
   where 
  \begin{equation*}
      c = \frac{1 - t^2}{t^4 + 3t^2}, \quad t \neq -1,0,1,
  \end{equation*}
   then the torsion subgroup of $E$ over $\Q(\sqrt{d})$ is $\Z/2\Z \times \Z/12\Z$.  

   In this case, obtaining $t$ from a rational solution of 
   \begin{equation*}
       dy^2 = x^3 - x^2 + x,
   \end{equation*}
   which is the $d$-quadratic twist of $X_1(2,12)$, is slightly more subtle. Following \cite{JeonKimLee}, we may replace $d(t)$ by $(t^2 - 1)(t^2 + 3)$ without changing the quadratic field $\Q(\sqrt{d})$. This leads to the hyperelliptic curve
   \begin{equation*}
       dz^2 = (t^2 - 1)(t^2 + 3),
   \end{equation*}
   for which we seek rational solutions $(t,z)$. This hyperelliptic curve is birationally equivalent to our elliptic curve via
   \begin{equation*} 
       t = \frac{x+1}{x}, \quad z = \frac{y}{x^2}.
   \end{equation*}
   Consequently, rational solutions to the elliptic curve correspond to rational elliptic curves with the desired torsion.  

   Since different $x$-values yield distinct $t$-values, they produce distinct $j$-invariants, giving infinitely many distinct rational elliptic curves. This is because each solution corresponds to a distinct quadratic point on $X_1(2,12)$, which in turn corresponds to a distinct elliptic curve with the desired torsion.
\end{proof}

Before starting our next proofs, which rely heavily on $2$-descent, we give a brief recap of the method following the notes of Bright \cite{2-descentpaper}.

Let $E/\Q$ be an elliptic curve. If $E$ admits a rational $2$-isogeny, then it must have a non-trivial $2$-torsion point. Our main goal is to compute the quotient $E(\Q)/2E(\Q)$, which determines the rank $\rk(E(\Q))$ once $E(\Q)[2]$ is known.

Assume that $E(\Q)[2]$ is non-trivial. Then $E$ can be written in the form
\begin{equation*}
    E: y^2 = x(x^2 + ax + b), \quad a,b \in \Q.
\end{equation*}
With this form, we have the following lemma, which constructs an explicit $2$-isogeny from $E$ to another elliptic curve $E'/\Q$.

\begin{lemma}\label{2isogenyconstruction}
    There exists an elliptic curve $E'$ over $\Q$ and an isogeny $\phi: E \to E'$ with kernel $\{O, (0,0)\}$. Specifically, $E'$ is given by
    \begin{equation*}
        E': y^2 = x(x^2 + a'x + b'), \quad a' = -2a, \; b' = a^2 - 4b,
    \end{equation*}
    and $\phi$ is defined by
    \begin{equation*}
        \phi(x,y) = \Bigl(x + a + \frac{b}{x}, \; y - \frac{by}{x^2}\Bigr),
    \end{equation*}
    for all $(x,y) \in E(\mathbb{C})$ with $(x,y) \neq (0,0)$, which is mapped to $O$.
\end{lemma}

From $E'$ we can define another isogeny $\phi'$ and elliptic curve $E''$:
\begin{equation*}
    E'': y^2 = x(x^2 + a'' x + b''), \quad a'' = -2a' = 4a, \quad b'' = (a')^2 - 4b' = 16b.
\end{equation*}
The curve $E''$ is isomorphic to $E$ via $x \mapsto 4x$ and $y \mapsto 8y$. This provides the dual isogeny $\hat{\phi}: E' \to E$ given by
\begin{equation*}
    \hat{\phi}(x,y) = \Bigl(\frac{1}{4}\bigl(x + a'b'/x\bigr), \; \frac{1}{8}\bigl(y - b'y/x^2\bigr)\Bigr).
\end{equation*}
The kernel of $\hat{\phi}$ is $\{O, (0,0)\} \subseteq E'(\Q)$, and the composite $\hat{\phi} \circ \phi: E \to E$ equals the multiplication-by-$2$ map.

Next, we study the image of $E(\Q)$ in $E'(\Q)$. In particular, we are interested in the quotient $E'(\Q)/\phi(E(\Q))$. To compute this, we define the following map:

\begin{lemma} \label{Q/Q^2}
    Define $\varphi: E'(\Q) \to \Q^\times/(\Q^\times)^2$ by
    \begin{align*}
        \varphi((x,y)) &= \overline{x}, \quad x \neq 0,\\
        \varphi((0,0)) &= \overline{a^2 - 4b},\\
        \varphi(O) &= \overline{1}.
    \end{align*}
    Then $\varphi$ is a homomorphism of groups, and the sequence
    \begin{equation*}
        E(\Q) \xrightarrow{\phi} E'(\Q) \xrightarrow{\varphi} \mathrm{\Kum}
    \end{equation*}
    is exact. In particular,
    \begin{equation*}
        \phi(E(\Q)) = \ker \varphi, \quad \text{and} \quad E'(\Q)/\phi(E(\Q)) \cong \varphi(E'(\Q)).
    \end{equation*}
\end{lemma}

Thus, computing the quotient $E'(\Q)/\phi(E(\Q))$ reduces to finding the image of $E'(\Q)$ in $\mathrm{\Kum}$. The following lemma gives an explicit criterion:

\begin{lemma}
    Let $r$ be a square-free integer. Then $\overline{r} \in \mathrm{\Kum}$ lies in the image of $\varphi$ if and only if the equation
    \begin{equation*}
        r^2 l^4 + a' r l^2 m^2 + b' m^4 = r n^2
    \end{equation*}
    has a non-zero solution $(l,m,n) \in \Z^3$. Moreover, this can occur only if $r \mid b'$, and one may assume $\gcd(l,m) = 1$.
\end{lemma}

We now present the last piece of the puzzle for computing $E(\Q)/2E(\Q)$. After this, we turn to calculations for the modular curves $X_1(2,10)$ and $X_1(2,12)$ and their $-p$-quadratic twists.

\begin{lemma}[Kernel-cokernel sequence]\label{kercoker}
    Let $A \xrightarrow{f} B \xrightarrow{g} C$ be homomorphisms of abelian groups. Then the sequence
    \begin{equation*}
        0 \to \ker f \xrightarrow{\mathrm{id}} \ker(g \circ f) \xrightarrow{f} \ker g \xrightarrow{\mathrm{id}} \coker f \xrightarrow{g} \coker(g \circ f) \xrightarrow{\mathrm{id}} \coker g \to 0
    \end{equation*}
    is exact.
\end{lemma}

\begin{proposition}\label{rank10}
 Let $E : y^2 = x(x^2 + x - 1)$
be the modular curve $X_1(2,10)$, and let $p \equiv 3 \pmod{4}$ be a prime.  
Let $E_{-p}$ denote the $(-p)$-quadratic twist of $E$, given by
\begin{equation*}
E_{-p} : y^2 = x(x^2 - px - p^2).
\end{equation*}
Then the Mordell--Weil rank of $E_{-p}$ over $\Q$ satisfies
\begin{equation*}
\rk(E_{-p}(\Q)) = 0 \quad \text{if } p \equiv 2,3 \pmod{5},
\end{equation*}
and
\begin{equation*}
0 \leq \rk(E_{-p}(\Q)) \leq 1 \quad \text{if } p \equiv 1,4 \pmod{5}.
\end{equation*}

\end{proposition}
\begin{proof}
 Applying $2$-descent via the $2$-isogeny between $E_{-p}$ and $E'_{-p}$ as 
 described in Lemma \ref{2isogenyconstruction}, we obtain
\begin{align*}
    E_{-p} &: y^2 = x(x^2 - px - p^2),\\
    E'_{-p} &: y^2 = x(x^2 + 2px + 5p^2).
\end{align*}

To determine the image $\varphi(E'_{-p}(\Q))$ in $\Kum$, we need to solve the 
equation
\begin{equation*}
    r^2l^4 + 2prl^2m^2 + 5p^2m^4 = rn^2,
\end{equation*}
where $r$ divides $5p^2$. Since $r$ is assumed to be square-free, it suffices to consider
\begin{equation*}
    r = \pm 1, \; \pm 5, \;\pm p,\; \pm 5p.
\end{equation*}
Because the left-hand side is always positive, we may discard the negative values. Moreover, by the definition of the map $\varphi$ described in Lemma \ref{Q/Q^2}, 
the classes $\overline{1}$ and $\overline{5}$ already lie in its image.
 Hence it remains to check the cases $r=p$ and $r=5p$.

If $p \mid r$, then $p^2$ divides the left-hand side of the equation, which implies that $p \mid n$. Writing $n = pk$, we proceed as follows.

If $r=p$, then we obtain
\begin{equation*}
    l^4 + 2l^2m^2 + 5m^4 = pk^2.
\end{equation*}
Reducing modulo $p$, this yields
\begin{equation*}
    (l^2 + m^2)^2 + 4m^4 \equiv 0 \pmod{p}.
\end{equation*}
At this point, the assumption $p \equiv 3 \pmod{4}$ becomes crucial. Since $-1$ is not a quadratic residue modulo $p$, the congruence implies that $p$ divides both $l$ and $m$. This contradicts the assumption $\gcd(l,m)=1$.

If $r=5p$, then we obtain
\begin{equation*}
    5l^4 + 2l^2m^2 + m^4 = 5pk^2.
\end{equation*}
Reducing modulo $p$, this yields
\begin{equation*}
    (l^2 + m^2)^2 + 4l^4 \equiv 0 \pmod{p}.
\end{equation*}
Since $-1$ is not a quadratic residue modulo $p$, the congruence implies that $p$ divides both $l$ and $m$. This contradicts the assumption $\gcd(l,m)=1$.

 Hence
\begin{equation*}
 E'_{-q}(\Q)/ \phi( E_{-q}(\Q)) \cong    \varphi(E'_{-q}(\Q))\cong \Z/2\Z.
\end{equation*}

We now turn to the computation of the image of $\varphi'$ in $\Kum$.  
In order to determine $\varphi'(E_{-p})$, we must solve the equation
\begin{equation*}
    r^2l^4 - prl^2m^2 - p^2m^4 = rn^2,
\end{equation*}
where $r$ divides $-p^2$. Since $r$ is square-free, it suffices to consider the cases
$r = \pm 1, \pm p$. We describe $\varphi'$ similarly to the definition of 
$\varphi$ in Lemma \ref{Q/Q^2}. By the definition of the map $\varphi'$, the 
classes $\overline{1}$ and $\overline{-1}$ lie in the image $\varphi'(E_{-p})$. 
Therefore, it remains to determine whether $\overline{p}$ or $\overline{-p}$ lies
in the image.

Since $\varphi'$ is a group homomorphism and $\overline{-1}$ already belongs to the image,
it is enough to check one of these two cases.

If $r = p$, then we obtain
\begin{equation*}
    p^2l^4 - p^2l^2m^2 - p^2m^4 = pn^2.
\end{equation*}
It follows that $p^2$ divides the left-hand side, and hence $p$ divides $n$. Setting $n = pk$, we obtain
\begin{equation*}
    l^4 - l^2m^2 - m^4 = pk^2.
\end{equation*}
Reducing modulo $p$, this yields
\begin{equation*}
    (2l^2 - m^2)^2 - 5 m^4 \equiv 0 \pmod{p}.
\end{equation*}
When $p \equiv \pm 2 \pmod{5}$, no solution exists because $5$ is not a quadratic residue modulo $p$. This establishes that the image of $\varphi'$ has exactly $2$ elements in this case, and by the kernel-cokernel sequence given in Lemma \ref{kercoker}, we conclude that $\rk_(E_{-p}(\Q)) = 0$.  

On the other hand, if $p \equiv \pm 1 \pmod{5}$, then the image of $\varphi'$ has either $2$ or $4$ elements, which implies
\begin{equation*}
    0 \leq \rk_(E_{-p}(\Q)) \leq 1.
\end{equation*}
In this case, our \texttt{SageMath} calculations suggest that the rank is exactly $1$, and 
we will prove this statement conditionally on the parity conjecture in Section 
\ref{conditionalsection}.
\end{proof}

\begin{corollary}\label{prime10}
   Let $p$ be a prime such that $p \equiv 3,7 \pmod{20}$. Then there does not exist an elliptic curve $E/\Q$ such that
\begin{equation*}
    E(\Q(\sqrt{-p}))_{\mathrm{tors}} \cong \Z/2\Z \times \Z/10\Z.
\end{equation*}
\end{corollary}

\begin{proposition}\label{multipleprimesfactored}
   Let $E: y^2 = x(x^2 + x - 1)$ be the equation of the modular curve $X_1(2,10)$.  
Let
\begin{equation*}
    d = \prod_{i=1}^n (-p_i)
\end{equation*}
for some integer $n \geq 1$, where $p_i$ are distinct primes satisfying $p_i \equiv 3,7 \pmod{20}$ for all $1 \leq i \leq n$.  
Let $E_d$ denote the $d$-quadratic twist of $E$, given by
\begin{equation*}
    E_d: y^2 = x(x^2 + dx - d^2).
\end{equation*}
Then
\begin{equation*}
    \rk(E_d(\Q)) = 0.
\end{equation*}

\end{proposition}

\begin{proof}
Applying $2$-descent using the $2$-isogeny between $E_{d}$ and $E'_{d}$ gives
\begin{align*}
     E_{d}:y^2 &= x(x^2+dx-d^2),\\
     E'_{d}:y^2 &= x(x^2-2dx+5d^2).
\end{align*}
To find the image $\varphi(E'_{d})$, we need to solve the equation
\begin{equation*}
    r^2l^4 - 2d r l^2 m^2 + 5 d^2 m^4 = r n^2,
\end{equation*}
where $r$ divides $5d^2$. By definition of the map, $\overline{1}$ and $\overline{5}$ belong to the image of $\varphi$. We will show that the image consists exactly of these two elements.  

Since the right-hand side is positive, $r>0$ must hold. Assume that a prime $p$ dividing $d$ also divides $r$. Write $d = pf$ and $r = ps$, yielding
\begin{equation*}
    p^2s^2l^4 - 2 p^2fsl^2m^2 + 5p^2f^2m^4 = psn^2.
\end{equation*}
The left-hand side must be divisible by $p^2$ since the right-hand side is divisible by $p^2$. Both $d$ and $r$ are square-free, so $p$ does not divide $f$ or $s$. Hence, $p$ must divide $n$; write $n = pk$ and divide both sides by $p^2$ to obtain
\begin{equation*}
   s^2l^4 - 2fs ^2m^2 + 5f^2m^4 = psk^2.  
\end{equation*}
The left-hand side can be rewritten as a sum of squares:
\begin{equation*}
     s^2l^4 - 2fsl^2m^2 + 5f^2m^4 = (sl^2 - fm^2)^2 + (2fm^2)^2. 
\end{equation*}
Using the assumption $p \equiv 3 \pmod{4}$, we deduce that $p$ divides both $s l^2 - f m^2$ and $2 f m^2$. Since $p$ does not divide $s$ or $f$, it follows that $p$ divides $l$ and $m$, contradicting $\gcd(l,m) = 1$.  

We conclude that $r$ is not divisible by any prime factor of $d$. As $r$ is a positive square-free integer dividing $5 d^2$, it follows that $r = 1$ or $r = 5$, so the image $\varphi(E'_{d})$ has exactly two elements.  

Next, to find the image $\varphi'(E_{d})$, we solve
\begin{equation*}
    r^2 l^4 + d r l^2 m^2 - d^2 m^4 = r n^2,
\end{equation*}
where $r$ divides $-d^2$. By definition, $\overline{1}$ and $\overline{-1}$ are in the image of $\varphi'$. We will show that these are the only elements.  

Assume that a prime $p$ dividing $d$ also divides $r$, and write $d = p f$, $r = p s$. Then
\begin{equation*}
    p^2 s^2 l^4 + p^2 f s l^2 m^2 - p^2 f^2 m^4 = p s n^2.
\end{equation*}
As before, $p^2$ divides the left-hand side, so $p$ must divide $n$; write $n = p k$ and divide both sides by $p^2$:
\begin{equation*}
   s^2 l^4 + f s l^2 m^2 - f^2 m^4 = p s k^2.  
\end{equation*}
The left-hand side can be rewritten as
\begin{equation*}
   (2 s l^2 + f m^2)^2 - 5 (f m^2)^2 = 4 s^2 l^4 + 4 f s l^2 m^2 - 4 f^2 m^4 = 4 p s k^2.  
\end{equation*}
The assumption $p \equiv \pm 2 \pmod{5}$ ensures that $5$ is not a quadratic residue modulo $p$, forcing $p$ to divide $2 s l^2 + f m^2$ and $f m^2$, hence also $s l^2$. Since $p$ does not divide $f$ or $s$, it must divide $l$ and $m$, contradicting $\gcd(l,m) = 1$. Therefore, $r = \pm 1$, and the image $\varphi'(E_{d})$ has exactly two elements.  

By the kernel-cokernel sequence given in Lemma \ref{kercoker}, we conclude that $ \mathrm{rk}(E_d(\Q)) = 0.$
\end{proof}

\begin{remark}
An analogous theorem does not hold for $X_1(2,12)$. As a counterexample,
we have $\mathrm{rk}_\Q(E_d) = 2$ for $d = -11 \cdot 59 \cdot 83$, as can be 
verified using our \texttt{SageMath} calculations.
\end{remark}   

\begin{corollary}\label{noZ2Z10overmultiple}
  Let $  d = \prod_{i=1}^n (-p_i)$
for some integer $n \geq 1$, where the primes $p_i$ satisfy $p_i \equiv 3,7 \pmod{20}$ for all $1 \leq i \leq n$. Then there does not exists an elliptic curve $E/\Q$ such that
\begin{equation*}
    E(\Q(\sqrt{d}))_{\mathrm{tors}} \cong \Z/2\Z \times \Z/10\Z.
\end{equation*}
\end{corollary}

\begin{proposition}\label{rank12}
 Let $ E: y^2 = x(x^2 - x + 1)$
be the equation of the modular curve $X_1(2,12)$. Let $p \equiv 11 \pmod{12}$ be a prime, and let $E_{-p}$ denote the $-p$-quadratic twist of $E$ given by
\begin{equation*}
    E_{-p}: y^2 = x(x^2 + px + p^2).
\end{equation*}
Then the Mordell--Weil rank of $E_{-p}$ over $\Q$ satisfies
\begin{equation*}
\rk(E_{-p}(\Q)) = 0 \quad \text{if } p \equiv 11 \pmod{24},
\end{equation*}
and
\begin{equation*}
0 \leq \rk(E_{-p}(\Q)) \leq 1 \quad \text{if } p \equiv 23 \pmod{24}.
\end{equation*}
\end{proposition}

\begin{proof}
    Applying $2$-descent using the $2$-isogeny between $E_{-p}$ and $E'_{-p}$ gives
\begin{align*}
    E_{-p}:& \ y^2 = x(x^2 + px + p^2),\\
    E'_{-p}:& \ y^2 = x(x^2 - 2px - 3p^2).
\end{align*}

To compute the image $\varphi'(E_{-p}(\Q))$, we need to solve
\begin{equation*}
    r^2l^4 + prl^2 m^2 + p^2 m^4 = rn^2
\end{equation*}
with $r \mid p^2$. It suffices to check $r = \pm 1, \pm p$. Since $\overline{1}$ belongs to the image and the left-hand side is always positive, we only need to consider $r=p$. In this case, $p \mid n$, so put $n = pk$ to obtain
\begin{equation*}
    l^4 + l^2 m^2 + m^4 = pk^2.
\end{equation*}
Assuming $\gcd(l,m)=1$, this equation requires $-3$ to be a quadratic residue modulo $p$. Since $p \equiv 2 \pmod{3}$, this is impossible. Hence, $\varphi'(E_{-p}(\Q))$ is trivial.

Next, for the image $\varphi(E'_{-p}(\Q))$, we solve
\begin{equation*}
    r^2l^4 - 2prl^2 m^2 - 3p^2 m^4 = rn^2
\end{equation*}
with $r \mid -3p^2$. We only need to consider $r = \pm 1, \pm 3, \pm p, \pm 3p$. By definition, $\overline{1}, \overline{-3}$ belong to the image. Furthermore, putting $l = m = 1$ and $n = 0$, we see that $\overline{-p}$ also belongs to the image. Since the image is a group, it suffices to check $r = p$, which determines whether $\rk(E_{-p}) = 0$ or $1$.

For $r = p$, let $n = pk$; then
\begin{equation*}
    l^4 - 2l^2 m^2 - 3 m^4 = (l^2 + m^2)(l^2 - 3 m^2) = pk^2.
\end{equation*}
Assuming $\gcd(l,m)=1$, we have $\gcd(l^2 + m^2, l^2 - 3m^2) = 1$ or $2$. If exactly one of $l$ or $m$ is even, the left-hand side is congruent modulo $4$ to either $l^4$ or $-3m^4$, which is always $1$, contradicting $p \equiv 3 \pmod{4}$. Therefore, both $l$ and $m$ are odd, giving
\begin{equation*}
    l^2 + m^2 = 2 u^2, \qquad l^2 - 3 m^2 = 2 p v^2
\end{equation*}
for some integers $u, v$. Parametrizing the first equation yields
\begin{equation*}
    \pm l = \alpha^2 - 2 \beta^2, \qquad \pm m = \alpha^2 - 4\alpha \beta + 2 \beta^2
\end{equation*}
with $\gcd(\alpha, \beta) = 1$. Using $p \mid l^2 - 3 m^2$, we have $l \equiv \omega m \pmod{p}$ for some $\omega$ with $\omega^2 \equiv 3 \pmod{p}$. Then
\begin{equation*}
    \alpha^2 (1 - \omega) + 4 \omega \alpha \beta - 2(1 + \omega) \beta^2 \equiv 0 \pmod{p}.
\end{equation*}
The discriminant of this quadratic is $32$, which must be a quadratic residue modulo $p$. Hence $p \equiv -1 \pmod{8}$ must hold, under the assumption $p \equiv 3 \pmod{4}$. This shows that if $p \equiv 11 \pmod{24}$, then $\rk(E_{-p}(\Q)) = 0$. If $p \equiv 23 \pmod{24}$, then $0 \leq \rk(E_{-p}(\Q)) \leq 1$.
\end{proof}

\begin{corollary}\label{prime12}
  Let $p$ be a prime such that $p\equiv 11 \pmod{24}$. Then there is no elliptic curve $E/\Q$ such that 
  \begin{equation*}
      E(\Q(\sqrt{-p}))_{\text{tors}}\cong \Z/2\Z \times \Z/12\Z.
  \end{equation*}
  \end{corollary}

\begin{proposition}
      Let $E:y^2=x(x^2-x+1)$ be the equation of the modular curve $X_1(2,12)$. Let $d=pq$ for some distinct primes $p,q$ satisfying $p,q\equiv 11 \pmod{24}$. Let $E_d$ be the $d$-quadratic twist of $E$ given by the equation 
      \begin{equation*}
          E_d:y^2=x(x^2-dx+d^2).
      \end{equation*}
      Then $rk(E_d(\Q))=0$.
\end{proposition}

\begin{proof}
Applying $2$-descent via the $2$-isogeny between $E_{d}$ and $E'_{d}$ gives
\begin{align*}
    E_{d}: y^2 &= x(x^2 - dx + d^2),\\
    E'_{d}: y^2 &= x(x^2 + 2dx - 3d^2).
\end{align*}
For the image $\varphi'(E_{d}(\mathbb{Q}))$, we need to solve
\begin{equation*}
    r^2 l^4 - d r l^2 m^2 + d^2 m^4 = r n^2
\end{equation*}
with $r \mid d^2$. Since $r$ is square-free, it must divide $d$. The left-hand side is always positive, so we have $r>0$. We know that $\overline{1}$ belongs to the image. We will show that $r=1$ is the only possibility.

Without loss of generality, assume that prime $p$ divides $r$, aiming for a contradiction. Write $r = ps$, where $s=1$ or $s=q$. Then
\begin{equation*}
    p^2 s^2 l^4 - p^2 q^2 l^2 m^2 + p^2 q^2 m^4 = ps n^2.
\end{equation*}
Observe that $p^2$ divides the left-hand side, but $p^2$ does not divide $ps$, so $p$ must divide $n$. Write $n = pk$, giving
\begin{equation*}
    s^2 l^4 - q^2 l^2 m^2 + q^2 m^4 = ps k^2.
\end{equation*}
We see that $p$ divides the right-hand side. Since $p \equiv 2 \pmod{3}$, $-3$ is not a quadratic residue modulo $p$, so $p$ must divide both $s l^2$ and $q m^2$. Recall that $s=1$ or $s=q$, and $q$ is a prime distinct from $p$, so $p$ must divide both $l$ and $m$. This contradicts $\gcd(l,m)=1$.
Hence, the image of $\varphi'(E_{-p}(\mathbb{Q}))$ is trivial.

For the image $\varphi(E'_{d}(\mathbb{Q}))$, we need to solve
\begin{equation*}
    r^2 l^4 + 2 d r l^2 m^2 - 3 d^2 m^4 = r n^2
\end{equation*}
with $r \mid 3 d^2$. By the definition of the map $\varphi$, we know that $\overline{1}$ and $\overline{-3}$ belong to the image. We will show that the image $\varphi'(E_{d}(\mathbb{Q}))$ is exactly the set $\{\overline{1},\overline{-3}\}$.  

Since $r$ is a square-free integer, and $[-3]$ belongs to the image, the primes $p$ and $q$ are symmetric. Thus, it suffices to check whether $r = p$ or $r = -p$ is possible.  

If $r = p$, then
\begin{equation*}
    p^2 l^4 + 2 p^2 q l^2 m^2 - 3 p^2 q^2 m^4 = p n^2.
\end{equation*}
The left-hand side is divisible by $p^2$, so $p$ must divide $n$. Writing $n = p k$ gives
\begin{equation*}
    l^4 + 2 q l^2 m^2 - 3 q^2 m^4 = p k^2.
\end{equation*}
Now we have two cases: either $q$ divides $l$ or it does not.  

If $q$ does not divide $l$, then modulo $q$ the left-hand side reduces to $l^4$, which is a quadratic residue modulo $q$. Since $q$ does not divide $l$, it also does not divide $k$, so $p$ must be a quadratic residue modulo $q$. However, both $p$ and $q$ are congruent to $3 \pmod{4}$, so by quadratic reciprocity, $q$ is not a quadratic residue modulo $p$, leading to a contradiction.

Factorizing the left-hand side, we get
\begin{equation*}
    (l^2 - q m^2)(l^2 + 3 q m^2) = p k^2.
\end{equation*}
Since $\gcd(l,m)=1$ and $q$ is not a quadratic residue modulo $p$, we see that $p$ does not divide $l^2 - q m^2$, so it must divide $l^2 + 3 q m^2$.  

Observe that the greatest common divisor of the two factors on the left can be either $1$, if one of $l$ and $m$ is odd and the other is even, or a power of $2$, if both $l$ and $m$ are odd. Furthermore, $l^2 + 3 q m^2$ is always positive.  

Thus, we have two possibilities:  
\begin{equation*}
    l^2 - q m^2 = u^2 \quad \text{and} \quad l^2 + 3 q m^2 = p v^2
\end{equation*}
or
\begin{equation*}
    l^2 - q m^2 = 2 u^2 \quad \text{and} \quad l^2 + 3 q m^2 = 2 p v^2
\end{equation*}
for some integers $u$ and $v$.  

The second possibility cannot occur because $l^2 - q m^2$ is a quadratic residue modulo $q$, while $2$ is not a quadratic residue modulo $q$ since $q \equiv 3 \pmod{8}$. The first possibility also fails because $l^2 + 3 q m^2$ can only take values $0,1,5 \pmod{8}$, whereas $p v^2$ can only take values $0,3,4 \pmod{8}$ since both $p$ and $q$ are congruent to $3 \pmod{8}$.  
This concludes the case when $q$ does not divide $l$.

If $q$ divides $l$, we write $l = q t$, giving
\begin{equation*}
    q^4 t^4 + 2 q^3 t^2 m^2 - 3 q^2 m^4 = p k^2.
\end{equation*}
Then $q$ divides $k$, so we put $k = q h$ to obtain
\begin{equation*}
    q^2 t^4 + 2 q t^2 m^2 - 3 m^4 = p h^2.
\end{equation*}
Factorizing the left-hand side, we get
\begin{equation*}
    (q t^2 - m^2)(q t^2 + 3 m^2) = p h^2.
\end{equation*}
Since $q$ divides $l$ and $\gcd(l,m) = 1$, we know that $q$ does not divide $m$. Then modulo $q$, the right-hand side is congruent to $-3 m^4$, which is not a quadratic residue since $q \equiv 2 \pmod{3}$. Therefore, $p$ is not a quadratic residue modulo $q$. By quadratic reciprocity, since both $p$ and $q$ are congruent to $3 \pmod{4}$, $q$ is a quadratic residue modulo $p$.  

This implies that $p$ does not divide $q t^2 + 3 m^2$ because $-3 q$ is not a quadratic residue modulo $p$, and $p$ cannot divide both $t$ and $m$ due to $\gcd(l,m)=1$. Hence, $p$ must divide $q t^2 - m^2$. Observe that the greatest common divisor of the two factors on the left can be either $1$, if one of $t$ and $m$ is odd and the other is even, or a power of $2$, if both are odd. Furthermore, $q t^2 + 3 m^2$ is always positive.  

Thus, we have two possibilities:
\begin{equation*}
    q t^2 - m^2 = p u^2 \quad \text{and} \quad q t^2 + 3 m^2 = v^2
\end{equation*}
or
\begin{equation*}
    q t^2 - m^2 = 2 p u^2 \quad \text{and} \quad q t^2 + 3 m^2 = 2 v^2
\end{equation*}
for some integers $u$ and $v$.  

The second possibility cannot occur because $q$ does not divide $m$ and $2$ is not a quadratic residue modulo $q$. The first possibility also fails because $q t^2 + 3 m^2 \equiv 3,6,7 \pmod{8}$, while $v^2 \equiv 0,1,4 \pmod{8}$.  
This concludes the case when $q$ divides $l$. Therefore, the case $r = p$ does not work.

 If $r=-p$, then we get
\begin{equation*}
    p^2 l^4 - 2 p^2 q l^2 m^2 - 3 p^2 q^2 m^4 = p n^2.
\end{equation*}
Since $p^2$ divides the left-hand side, $p$ divides $n$. Put $n = p k$ to obtain
\begin{equation*}
    l^4 - 2 q l^2 m^2 - 3 q^2 m^4 = p k^2.
\end{equation*}
Factorizing the left-hand side gives
\begin{equation*}
    (l^2 + q m^2)(l^2 - 3 q m^2) = p k^2.
\end{equation*}

We consider two subcases: either $q$ divides $l$ or $q$ does not divide $l$.

If $q$ does not divide $l$, then the greatest common divisor of the two factors is either $1$, if one of $l$ and $m$ is odd and the other even, or a power of $2$, if both are odd. Moreover, modulo $q$, the left-hand side is congruent to $l^4$, and $q$ does not divide $k$, so $p$ is a quadratic residue modulo $q$. By quadratic reciprocity (since $p, q \equiv 3 \pmod{4}$), $q$ is a quadratic residue modulo $p$, so $-q$ is not. Therefore, $p$ must divide the second factor, $l^2 - 3 q m^2$, because it cannot divide the first factor.  

Since $l^2 + q m^2$ is positive, we have two possibilities:
\begin{equation*}
    l^2 + q m^2 = u^2 \quad \text{and} \quad l^2 - 3 q m^2 = p v^2
\end{equation*}
or
\begin{equation*}
    l^2 + q m^2 = 2 u^2 \quad \text{and} \quad l^2 - 3 q m^2 = 2 p v^2
\end{equation*}
for some integers $u,v$.  

The second option cannot occur because $q$ does not divide $l$ and $2$ is not a quadratic residue modulo $q$. For the first option, solving $l^2 + q m^2 = u^2$ with $\gcd(l,m) = 1$, we can parametrize all solutions as
\begin{equation*}
    \pm l = \alpha^2 - q \beta^2, \quad \pm u = \alpha^2 + q \beta^2, \quad \pm m = 2 \alpha \beta
\end{equation*}
if one of $\alpha, \beta$ is even, or
\begin{equation*}
    \pm l = \frac{\alpha^2 - q \beta^2}{2}, \quad \pm u = \frac{\alpha^2 + q \beta^2}{2}, \quad \pm m = \alpha \beta
\end{equation*}
if both $\alpha$ and $\beta$ are odd, with $\gcd(\alpha, \beta) = 1$.  

Plugging into the second equation, we get
\begin{equation*}
    l^2 - 3 q m^2 = (\alpha^2 - q \beta^2)^2 - 12 q \alpha^2 \beta^2 = p v^2,
\end{equation*}
which can be reorganized as
\begin{equation*}
    (\alpha^2 + q \beta^2)^2 - 16 q \alpha^2 \beta^2 = p v^2.
\end{equation*}

At least one of $\alpha$ or $\beta$ must be odd. If one is even, then $(\alpha^2 + q \beta^2)^2 \equiv 1 \pmod{8}$, so $v$ must be odd, giving $p v^2 \equiv 3 \pmod{8}$, a contradiction. Thus, both $\alpha$ and $\beta$ are odd. Then $(\alpha^2 + q \beta^2)^2 \equiv 16 \pmod{128}$, and $-16 q \alpha^2 \beta^2 \equiv -48 \pmod{128}$, so $p v^2 \equiv -32 \pmod{128}$. This requires $v$ divisible by $8$, making $p v^2$ divisible by $64$, while the left-hand side is not. Contradiction.  

If $q$ divides $l$, then we write $l = q t$ and $k = q K$, giving
\begin{equation*}
    (q t^2 + m^2)(q t^2 - 3 m^2) = p K^2.
\end{equation*}
Then $q \nmid m$. The gcd of the two factors is either $1$ (if one of $t, m$ is odd and the other even) or a power of $2$ (if both are odd). Moreover, modulo $q$, the left-hand side is $-3 m^4$, which is not a quadratic residue since $q \equiv 2 \pmod{3}$, and $q \nmid K$, so $p$ is not a quadratic residue modulo $q$. By quadratic reciprocity, $q$ is a quadratic residue modulo $p$, so $-q$ is not, and $p$ must divide the second factor.  

We then have
\begin{equation*}
    q t^2 + m^2 = u^2, \quad q t^2 - 3 m^2 = p v^2
\end{equation*}
or
\begin{equation*}
    q t^2 + m^2 = 2 u^2, \quad q t^2 - 3 m^2 = 2 p v^2.
\end{equation*}
The second option is impossible since $q \nmid m$ and $2$ is not a quadratic residue modulo $q$.  

For the first option, at least one of $t$ or $m$ is odd. Checking modulo $4$ and $8$ leads to contradictions in all subcases (either $t$ even or $m$ even), forcing both $t$ and $m$ to be odd. Then $u$ and $v$ are even. Writing $u = 2 U$ and $v = 2 V$, we obtain
\begin{equation*}
    4 m^2 + p v^2 = u^2 \quad \implies \quad m^2 + p V^2 = U^2.
\end{equation*}
Parametrizing this with $\gcd(\alpha,\beta) = 1$ and both $\alpha, \beta$ odd, we get
\begin{equation*}
    q t^2 = u^2 - m^2 = (\alpha^2 + p \beta^2)^2 - \alpha^2 \beta^2,
\end{equation*}
which is a contradiction modulo $8$.  

Thus, both $r = p$ and $r = -p$ cases do not work. Therefore, the image $\varphi(E_d'(\mathbb{Q}))$ contains exactly two elements. Using the kernel-cokernel sequence given in Lemma \ref{kercoker}, we conclude that $\rk(E_d(\mathbb{Q})) = 0.$
\end{proof}

\begin{corollary} \label{twoprimes12}
   Let $p$ and $q$ be distinct primes satisfying $p, q \equiv 11 \pmod{24}$. Then there does not exist an elliptic curve $E/\mathbb{Q}$ such that
\begin{equation*}
    E(\mathbb{Q}(\sqrt{pq}))_{\mathrm{tors}} \cong \mathbb{Z}/2\mathbb{Z} \times \mathbb{Z}/12\mathbb{Z}.
\end{equation*}
\end{corollary}

\section{Conditional Proofs of our Conjectures}\label{conditionalsection}
In this chapter, we give proofs of Conjectures \ref{con10} and \ref{con12}, 
conditional on the finiteness of the $2$-primary part of the Tate–Shafarevich 
group $\Sha(E/\mathbb{Q})$ for rational elliptic curves. We are grateful to 
Dokchitser \cite{Dokchitsernotes} for his notes on the parity conjecture.

Before using the results on the parity conjecture to prove our conjectures 
conditionally, we recall some known results about the parity conjecture. 
Firstly, we give a brief overview following the introduction of Radziwill and 
Soundararajan \cite{radsou}.  

Let $E$ be an elliptic curve defined over $\mathbb{Q}$ with conductor $N$. 
Write the associated Hasse-Weil $L$-function as
\begin{equation*}
    L(E/\mathbb{Q},s) = \sum_{n\geq 1} \frac{a(n)}{n^s},
\end{equation*}
where the coefficients are normalized such that the Hasse bound reads $|a(n)| \leq d(n)$ for all $n$, 
and so the center of the critical strip is $s=1/2$. Recall that $L(E/\mathbb{Q},s)$ has an analytic continuation
to the entire complex plane and satisfies the functional equation
\begin{equation*}
    \Lambda(E/\mathbb{Q},s) = w(E/\mathbb{Q}) \Lambda(E/\mathbb{Q},1-s),
\end{equation*}
where $w(E/\mathbb{Q})$, the root number, is $\pm1$, and
\begin{equation*}
    \Lambda(E/\mathbb{Q},s) = \Big(\frac{\sqrt{N}}{2\pi}\Big)^s \Gamma(s+\tfrac{1}{2}) L(E/\mathbb{Q},s).
\end{equation*}

Throughout this section, let $d$ be a fundamental discriminant coprime to $2N$, and 
let $\chi_d = \left(\frac{d}{\cdot}\right)$ denote the associated primitive quadratic character. 
Let $E_d$ denote the quadratic twist of $E$ by $d$. The twisted $L$-function is
\begin{equation*}
    L(E_d/\mathbb{Q},s) = \sum_{n\geq 1} \frac{a(n)\chi_d(n)}{n^s}.
\end{equation*}
If $\gcd(d, N) = 1$, then $E_d$ has conductor $N d^2$, and the completed $L$-function
\begin{equation*}
     \Lambda(E_d/\mathbb{Q},s) = \Big(\frac{\sqrt{N}|d|}{2\pi}\Big)^s \Gamma(s+\tfrac{1}{2}) L(E_d/\mathbb{Q},s)
\end{equation*}
is entire and satisfies the functional equation
\begin{equation*}
       \Lambda(E_d/\mathbb{Q},s) = w(E_d/\mathbb{Q}) \Lambda(E_d/\mathbb{Q},1-s),
\end{equation*}
with 
\begin{equation*}
    w(E_d/\mathbb{Q})  = \chi_d(-N) w(E/\mathbb{Q}).
\end{equation*}

Note that the sign of the functional equation determines the order of vanishing of the 
$L$-function at $s=1$, i.e., the analytic rank of the elliptic curve. It is a standard result that
\begin{equation*}
     (-1)^{\Rk(E(\mathbb{Q}))} = w(E/\mathbb{Q}),
\end{equation*}
where $\Rk$ denotes the analytic rank. The parity conjecture for rational elliptic curves states that
\begin{equation*}
    (-1)^{\rk(E(\mathbb{Q}))} = w(E/\mathbb{Q}),
\end{equation*}
so that
\begin{equation*}
    \Rk(E(\Q)) \equiv \rk(E(\mathbb{Q})) \pmod{2}.
\end{equation*}

This relationship between the analytic rank and quadratic twists allows us to determine 
when the rank is odd or even in our applications. For the elliptic curves $X_1(2,10)$ and 
$X_1(2,12)$, a standard \texttt{SageMath} calculation shows that their algebraic rank is zero. 
Under the parity conjecture, this implies that their analytic rank is even; indeed, the 
LMFDB confirms that their analytic rank is zero. By calculating $\chi_d(-N)$ for the discriminants 
$d$ of interest, we can determine the parity of the algebraic rank conditionally on the parity conjecture. 
Here, $N=20$ and $N=24$, respectively.

We will show that $\chi_d(-N)=-1$ for the $d$ given in Conjectures \ref{con10} and \ref{con12}, 
conditionally implying that the algebraic rank is odd. Previously, we showed that the algebraic 
rank is either $0$ or $1$ in these cases. Therefore, we do not need to assume the full Birch 
and Swinnerton-Dyer conjecture, which asserts equality of the algebraic and analytic ranks. 
Instead, it suffices to assume the parity conjecture, which asserts equality of the parities of 
these ranks.

Monsky \cite{Monsky} proved the parity conjecture for all rational elliptic curves, 
conditional on the finiteness of the $2$-primary part of the Tate–Shafarevich group 
$\Sha(E/\mathbb{Q})$. Dokchitser and Dokchitser \cite{Dokchitserparity} proved the 
parity conjecture for elliptic curves over number fields $E/K$, conditional on the finiteness 
of the $2$- and $3$-parts of $\Sha(E/K(E[2]))$, where $K(E[2])$ denotes the extension of $K$ 
obtained by adjoining the coordinates of the $2$-torsion points of $E$. In our setting, we 
do not assume the extra finiteness of the $3$-part, since we can reduce the problem 
to determining the parity of the rank for rational elliptic curves.

Therefore, in our application, we can prove our conjectures conditionally either by assuming 
the parity conjecture, or the finiteness of the $2$-primary part of $\Sha$, which implies the parity conjecture for rational elliptic curves.

\begin{proposition}[Conditional on the Parity Conjecture]
Let $p \equiv 11,19 \pmod{20}$ be a prime. Let 
\begin{equation*}
    E: y^2 = x(x^2 + x - 1)
\end{equation*}
be a model for the modular curve $X_1(2,10)$. Then $  \rk(E_{-p}(\mathbb{Q})) = 1.$
\end{proposition}

\begin{proof}
The conductor of $E$ is $N=20$. Since $p \equiv 3 \pmod{4}$, the discriminant of the quadratic field $\mathbb{Q}(\sqrt{-p})$ is $  d=-p.$

We have $w(E/\mathbb{Q}) = 1$, so to show $w(E_{-p}/\mathbb{Q}) = -1$, it suffices to check that 
\begin{equation*}
    \chi_{-p}(-20) = -1.
\end{equation*}
Recall that the primitive quadratic character is the Kronecker symbol. Since $20 = 2^2 \cdot 5$ and $p$ is odd, we have
\begin{equation*}
    \Big(\frac{-p}{-20}\Big) = \Big(\frac{-p}{-5}\Big),
\end{equation*}
and 
\begin{equation*}
    \Big(\frac{-p}{-5}\Big) = - \Big(\frac{-p}{5}\Big),
\end{equation*}
because $-p < 0$. 

By assumption, $p \equiv \pm 1 \pmod{5}$, so $-p$ is a quadratic residue modulo $5$. Therefore,
\begin{equation*}
    \chi_{-p}(-20) = - \Big(\frac{-p}{5}\Big) = -1.
\end{equation*}
This shows that $ \Rk(E_{-p}(\mathbb{Q}))$
is odd. Moreover, by Proposition \ref{rank10}, we know that
\begin{equation*}
    0 \leq \rk(E_{-p}(\mathbb{Q})) \leq 1.
\end{equation*}

If we assume the parity conjecture, or the finiteness of the $2$-primary part of $\Sha$, which implies the parity conjecture for rational elliptic curves, then for rational elliptic curves,
\begin{equation*}
    \Rk(E_{-p}(\mathbb{Q})) \equiv \rk(E_{-p}(\mathbb{Q})) \pmod{2}.
\end{equation*}
Since $\Rk(E_{-p}(\mathbb{Q}))$ is odd and the rank is at most $1$, it follows that
\begin{equation*}
    \rk(E_{-p}(\mathbb{Q})) = 1.
\end{equation*}

Combining this with Proposition \ref{infinitely10} gives the desired result.
\end{proof}

\begin{proposition}[Conditionally on Parity Conjecture] 
    Let $p\equiv 23 \pmod{24}$ be a prime. Let 
    \begin{equation*}
        E: y^2 = x(x^2 - x + 1)
    \end{equation*}
    be the model for the modular curve $X_1(2,12)$. Then $  \rk(E_{-p}(\mathbb{Q})) = 1.$
\end{proposition}

\begin{proof}
The conductor of $E$ is $N=24$. Since $p \equiv 3 \pmod{4}$, the discriminant of the quadratic field $\mathbb{Q}(\sqrt{-p})$ is $    d = -p.$

We know that $w(E/\mathbb{Q}) = 1$, so in order to show $w(E_{-p}/\mathbb{Q}) = -1$, it suffices to check that 
\begin{equation*}
    \chi_{-p}(-24) = -1.
\end{equation*}
The primitive quadratic character is the Kronecker symbol, so we calculate
\begin{equation*}
    \Big(\frac{-p}{-24}\Big) = \Big(\frac{-p}{-6}\Big),
\end{equation*}
and  
\begin{equation*}
    \Big(\frac{-p}{-6}\Big) = - \Big(\frac{-p}{6}\Big),
\end{equation*}
since $-p < 0$. The assumption $p \equiv -1 \pmod{3}$ implies that $-p$ is a quadratic residue modulo $3$, so
\begin{equation*}
    -\Big(\frac{-p}{6}\Big) = -\Big(\frac{-p}{3}\Big)\Big(\frac{-p}{2}\Big) = - \Big(\frac{-p}{2}\Big).
\end{equation*}
Finally, the assumption $p \equiv -1 \pmod{8}$ implies that 
\begin{equation*}
    \Big(\frac{-p}{2}\Big) = 1,
\end{equation*}
so
\begin{equation*}
    \chi_{-p}(-24) = - \Big(\frac{-p}{2}\Big) = -1.
\end{equation*}

This shows that $\Rk(E_{-p}(\mathbb{Q}))$ is odd. Moreover, by Proposition \ref{rank12}, we have 
\begin{equation*}
    0 \leq \rk(E_{-p}(\mathbb{Q})) \leq 1.
\end{equation*}

If we assume the parity conjecture, or the finiteness of the $2$-primary part of $\Sha$, which implies the parity conjecture for rational elliptic curves, then for rational elliptic curves we must have
\begin{equation*}
    \Rk(E_{-p}(\mathbb{Q})) \equiv \rk(E_{-p}(\mathbb{Q})) \pmod{2}.
\end{equation*}

This implies
\begin{equation*}
    \rk(E_{-p}(\mathbb{Q})) = 1.
\end{equation*}

Combining this with Proposition \ref{infinitely12} gives the desired result.
\end{proof}

Even though our results are conditional on certain conjectures, we have 
exhibited explicit families of quadratic fields over which there exist 
infinitely many rational elliptic curves whose torsion subgroups are 
isomorphic to 
\begin{equation*}
    \mathbb{Z}/2\mathbb{Z} \times \mathbb{Z}/10\mathbb{Z} \quad \text{or} \quad 
    \mathbb{Z}/2\mathbb{Z} \times \mathbb{Z}/12\mathbb{Z}.
\end{equation*}
This, in turn, yields an infinite family of primes $p$ for which the 
desired torsion subgroups occur over the cyclotomic fields $  \mathbb{Q}(\zeta_p).$

 \section{Eliminating Possible Torsion} \label{eliminationsection}
In this section, we present several methods for eliminating torsion subgroups
that cannot be realized. Our main goal is to improve some of the elimination
results obtained earlier \cite{Omer}.

We begin with an important lemma. It is a minor improvement of \cite[Lemma 3.1]{Omer}, and its proof is essentially the same; however, it is necessary for the results in the next section. The only difference is that here we apply the same argument to $\Gal(L/K)$ when considering 
the $L$-torsion of $E/K$, instead of $\Gal(K/\Q)$ when considering the $K$-torsion of $E/\Q$.

\begin{lemma}\label{upgradedfourpointpairing}
 Let $K$ be a number field and let $L/K$ be a Galois extension. Let $E/K$ be an elliptic curve. Let $q$ be an odd prime power, and suppose 
\begin{equation*}
    E(L)[q] \cong \mathbb{Z}/q\mathbb{Z}.
\end{equation*}
Assume that all elements of $\Gal(L/K)$ have finite order.
For any $\sigma \in \Gal(L/K)$, let $\ord(\sigma)$ denote the order of $\sigma$. 

\begin{enumerate}[(1)]
    \item If $\gcd(\phi(q), \ord(\sigma)) = 1$ for all $\sigma \in \Gal(L/K)$, then
    \begin{equation*}
        E(L)[q] = E(K)[q].
    \end{equation*}
    
    \item If $\gcd(\phi(q), \ord(\sigma)) \leq 2$ for all $\sigma \in \Gal(L/K)$, then
    \begin{equation*}
        E(L)[q] \cong E_d(L)[q] = E_d(K)[q]
    \end{equation*}
    for some $d \in K$ satisfying $\sqrt{d} \in L$.
\end{enumerate}

\end{lemma}

This lemma is necessary to absolve the result from depending on the degree of a 
Galois extension and instead state it in terms of the orders of its elements. It 
is also useful that we generalize our results to elliptic curves $E/K$ over an 
arbitrary number field $K$.

The lemma is particularly helpful when dealing with cyclotomic extensions of 
$\mathbb{Q}$: if $n$ is a positive integer not divisible by $16$, 
and every odd prime dividing $n$ is of the form $4k+3$, then the extension degree
$[\mathbb{Q}(\zeta_n):\mathbb{Q}]$ can be divisible by $4$, yet 
$\Gal(\mathbb{Q}(\zeta_n)/\mathbb{Q})$ contains no elements of order $4$. 
It is also very useful when considering the compositum of quadratic fields. 

In the next two results, we present direct applications of this lemma. It will 
also play a key role later in the proof of Theorem \ref{metabelian}, and in 
generalizing some of the results from \cite{Omer2}.

\begin{lemma}\label{13eliminationlem}
    Let $E/\Q$ be an elliptic curve, and let $L$ be a Galois number field such that 
    $\Gal(L/\Q)$ contains no element of order $3$ or $4$. 
    Then 
    \begin{equation*}
        E(L)_{\text{tors}} \not\cong \Z/13\Z.
    \end{equation*}    
\end{lemma}
\begin{proof}
    By Lemma \ref{upgradedfourpointpairing}, if $E(L)[13]\cong \Z/13\Z$, then there exists a square-free integer $d$ with $\sqrt{d}\in L$ such that 
    \begin{equation*}
        E_d(\Q)[13] \cong \Z/13\Z.
    \end{equation*}
    But $E_d$ is an elliptic curve over $\Q$, and by Mazur's theorem it cannot have a torsion point of order $13$. This gives a contradiction.
\end{proof}

\begin{corollary}\label{13elimination}
    Let $E/\Q$ be an elliptic curve, and let $n$ be a positive integer that is not divisible by $16$, and such that every odd prime dividing $n$ is of the form $12k+11$. Then
    \begin{equation*}
        E(\Q(\zeta_n))_{\text{tors}} \not\cong \Z/13\Z.
    \end{equation*}    
\end{corollary}
\begin{proof}
    Since $\Gal(\Q(\zeta_n)/\Q)\cong (\Z/n\Z)^\times$, the assumptions imply that the Galois group contains no element of order $3$ or $4$. The result then follows from Lemma \ref{13eliminationlem}.
\end{proof}

\begin{lemma}\label{25eliminationlem}
    Let $E/\Q$ be an elliptic curve, and let $L$ be a Galois number field such that 
    $\Gal(L/\Q)$ contains no element of order $4$ or $5$. 
    Then 
    \begin{equation*}
        E(L)_{\text{tors}} \not\cong \Z/25\Z.
    \end{equation*}     
\end{lemma}
\begin{proof}
    By Lemma \ref{upgradedfourpointpairing}, if $E(L)[25]\cong \Z/25\Z$, then there exists a square-free integer $d$ with $\sqrt{d}\in L$ such that
    \begin{equation*}
        E_d(\Q)[25] \cong \Z/25\Z.
    \end{equation*}
    But $E_d$ is an elliptic curve over $\Q$, and Mazur's theorem rules out torsion of order $25$. This is a contradiction.
\end{proof}

\begin{corollary}\label{25elimination}
    Let $E/\Q$ be an elliptic curve, and let $n$ be a positive integer that is not divisible by $16$, and such that every odd prime $p$ dividing $n$ satisfies $p \equiv 3 \pmod{4}$ and $p-1$ is not divisible by $5$. Then
    \begin{equation*}
        E(\Q(\zeta_n))_{\text{tors}} \not\cong \Z/25\Z.
    \end{equation*} 
\end{corollary}
\begin{proof}
    Since $\Gal(\Q(\zeta_n)/\Q)\cong (\Z/n\Z)^\times$, the assumptions guarantee that the Galois group contains no element of order $4$ or $5$. The result then follows from Lemma \ref{25eliminationlem}.
\end{proof}

We now present two important results that are used frequently in the classification of torsion subgroups.

\begin{theorem}[Fricke, Kenku, Klein, Kubert, Ligozat, Mazur, and Ogg, among others] \label{rationalisogeny}
    Let $E/\Q$ be an elliptic curve with an $n$-isogeny. Then 
    \begin{equation*}
        n \leq 19 \quad \text{or} \quad n \in \{21, 25, 27, 37, 43, 67, 163\}.
    \end{equation*}
    Moreover, if $E$ does not have complex multiplication, then 
    \begin{equation*}
        n \leq 18 \quad \text{or} \quad n \in \{21, 25, 37\}.
    \end{equation*}
\end{theorem}

\begin{lemma}[Chou, Lemma 2.7, \cite{Chou}]\label{n-isogeny}
    Let $L$ be a Galois extension of $\Q$, and let $E$ be an elliptic curve over $\Q$. If 
    \begin{equation*}
        E(L)_{\text{tors}} \cong \Z/m\Z \times \Z/mn\Z,
    \end{equation*}
    then $E$ admits an $n$-isogeny over $\Q$.
\end{lemma}

Next, we present a theorem that allows us to eliminate certain torsion subgroups 
when working with cyclotomic extensions $\Q(\zeta_n)$. We treated 
the case when $n$ is prime \cite{Omer}. Building upon that, we can now obtain results for 
composite $n$. This is an important step for the classification results presented 
immediately afterward.

Our result can also be generalized to an arbitrary abelian number field $L$ by 
considering the minimal positive integer $n$ such that $L \subset \Q(\zeta_n)$. 
The group structures we eliminate are precisely those that can appear as torsion 
subgroups over arbitrary abelian number fields. The complete list is given in
\cite[Theorem 2.3]{Omer}, which follows directly from \cite[Theorem 1.2]{Chou}.

\begin{theorem}\label{classificationtheoremmain}
    Let $E/\mathbb{Q}$ be an elliptic curve, and let $n$ be a positive integer. Then:
    \begin{enumerate}[(i)]
        \item If $n$ is not divisible by $3$, then 
        \begin{equation*}
            E(\Q(\zeta_n))_{\text{tors}} \not\cong \mathbb{Z}/N\mathbb{Z} \text{ for } N=21,27,
        \end{equation*}
        and 
        \begin{equation*}
            E(\Q(\zeta_n))_{\text{tors}} \not\cong \mathbb{Z}/3\mathbb{Z} \times \mathbb{Z}/3N\mathbb{Z} \text{ for } N=1,2,3,
        \end{equation*}
        and 
        \begin{equation*}
            E(\Q(\zeta_n))_{\text{tors}} \not\cong \mathbb{Z}/6\mathbb{Z} \times \mathbb{Z}/6\mathbb{Z}.
        \end{equation*}
        
        \item If $n$ is not divisible by $4$, then
        \begin{equation*}
            E(\Q(\zeta_n))_{\text{tors}} \not\cong \mathbb{Z}/4\mathbb{Z} \times \mathbb{Z}/4N\mathbb{Z} \text{ for } N=1,2,3,4,
        \end{equation*}
        and 
        \begin{equation*}
            E(\Q(\zeta_n))_{\text{tors}} \not\cong \mathbb{Z}/8\mathbb{Z} \times \mathbb{Z}/8\mathbb{Z}.
        \end{equation*}
        
        \item If $n$ is not divisible by $5$, then
        \begin{equation*}
            E(\Q(\zeta_n))_{\text{tors}} \not\cong \mathbb{Z}/N\mathbb{Z} \text{ for } N=15,17,37,
        \end{equation*}
        and 
        \begin{equation*}
            E(\Q(\zeta_n))_{\text{tors}} \not\cong \mathbb{Z}/5\mathbb{Z} \times \mathbb{Z}/5\mathbb{Z}.
        \end{equation*}
        
        \item If $n$ is not divisible by $3$ and $5$, then 
        \begin{equation*}
            E(\Q(\zeta_n))_{\text{tors}} \not\cong \mathbb{Z}/15\mathbb{Z}.
        \end{equation*}
        
        \item If $n$ is not divisible by $7$, then 
        \begin{equation*}
            E(\Q(\zeta_n))_{\text{tors}} \not\cong \mathbb{Z}/N\mathbb{Z} \text{ for } N=14,37.
        \end{equation*}
        
        \item If $n$ is not divisible by $7$, and $\phi(n)$ is not divisible by $3$, then 
        \begin{equation*}
            E(\Q(\zeta_n))_{\text{tors}} \not\cong \mathbb{Z}/2\mathbb{Z} \times \mathbb{Z}/14\mathbb{Z}.
        \end{equation*}
        
        \item If $n$ is not divisible by the prime $\ell \in \{11,17,19,43,67,163\}$, then 
        \begin{equation*}
            E(\Q(\zeta_n))_{\text{tors}} \not\cong \mathbb{Z}/\ell\mathbb{Z}.
        \end{equation*}
    \end{enumerate}
\end{theorem}

\begin{proof}
    Groups of the form $\Z/m\Z \times \Z/mN\Z$ can be eliminated using the Weil 
    pairing: if an elliptic curve defined over a number field $K$ has full 
    $m$-torsion over $K$, then $K$ must contain all $m$-th roots of unity. 
    This argument covers all cases where $m \ge 3$ in items (i)--(iii).

    Groups of the form $\Z/N\Z$ can be eliminated as follows. Let $E/\Q$ be an 
    elliptic curve and $K$ a Galois number field. If 
    \begin{equation*}
        E(K)_{\text{tors}} \cong \Z/m\Z \times \Z/mN\Z,
    \end{equation*} 
    then $E$ has a rational $N$-isogeny by \cite[Lemma 2.7]{Chou}. In the 
    statement of this theorem, the listed $N$ are special because there are only
    finitely many elliptic curves (finitely many rational $j$-invariants, up to 
    quadratic twist) with rational $N$-isogeny for each $N$. Moreover, the field 
    of definition of the $N$-isogeny is classified
    \cite[Table 2]{Omer}. In particular, the $x$-coordinates of the kernel points of the 
    $N$-isogeny are defined over these fields, and this does not change for 
    quadratic twists. Therefore, this classification puts a hard bound on which 
    cyclotomic fields $\Q(\zeta_n)$ can contain a torsion point of order $N$.

    These two arguments cover all group structures in the theorem except $\Z/2\Z \times \Z/14\Z$. For this case, the extra condition that $\phi(n)$ is not 
    divisible by $3$ implies that the extension degree $[\Q(\zeta_n):\Q]$ is not 
    divisible by $3$, so $\Q(\zeta_n)$ does not contain a cubic subfield. Since 
    $E$ has full $2$-torsion over $\Q(\zeta_n)$, the field of definition of the 
    $2$-torsion cannot have degree $3$ or $6$, so at least one $2$-torsion point 
    is rational. By \cite[Lemma 2.7]{Chou}, this implies that $E$ has a rational
    $7$-isogeny. Combining this with the rational $2$-torsion point, $E$ would 
    have a rational $14$-isogeny, which in turn forces $n$ to be divisible by $7$, as analyzed in 
    \cite{Omer}. This contradicts the assumption.
\end{proof}

\section{More Classification Results} \label{torsionsection}

In this chapter, we study torsion subgroups of elliptic curves over certain 
metabelian extensions, such as Kummer extensions of cyclotomic fields, building on 
and refining the classification results \cite{Omer}. We then generalize some 
results on $\Z_p$-extensions of number fields \cite{Omer2}, showing that 
torsion over the full $\Z_p$-extension is often already defined over the base 
field. 

In this context, our previous results on the impossibility of realizing the 
torsion subgroups $\Z/2\Z \times \Z/10\Z$ and $\Z/2\Z \times \Z/12\Z$ over 
quadratic fields imply that these groups also cannot appear over the corresponding 
metabelian or $\Z_p$-extensions. This connects the explicit elimination results 
over quadratic fields to broader infinite abelian extensions and illustrates the 
reach of our methods.

\begin{theorem}\label{4exceptions}
Let $E/\Q$ be an elliptic curve. Define the field
\begin{equation*}
L = \Q\Big(\{\mu_{p^\infty}: p>11 \text{ is a prime satisfying } p\equiv 11\Mod{12} \text{ and } p \not\equiv 1 \Mod{5} \}\Big). 
\end{equation*}
Then $E(L)_{\text{tors}}$ is isomorphic either to one of the groups from Mazur's theorem or to one of the following:
\begin{equation*}
\Z/2\Z \times \Z/10\Z, \quad \Z/2\Z \times \Z/12\Z, \quad \Z/16\Z, \quad \text{or} \quad \Z/2\Z \times \Z/16\Z.
\end{equation*}
\end{theorem}

\begin{proof}
It suffices to prove the statement for all cyclotomic fields $\Q(\zeta_n)\subset L$, since $E(L)_{\text{tors}}$ is finite \cite[Theorem 1.2]{Chou}, and the field of definition of this torsion subgroup is an abelian number field, so it must be contained in some $\Q(\zeta_n) \subset L$.

Let $n$ be a positive integer such that $\Q(\zeta_n) \subset L$. Then $E(\Q(\zeta_n))_{\text{tors}}$ is contained in the finite list of possible torsion groups given \cite[Theorem 2.3]{Omer}. Using Lemmas \ref{13elimination}, \ref{25elimination}, and Theorem \ref{classificationtheoremmain}, we can eliminate all groups except the following:
\begin{equation*}
E(\Q(\zeta_n))_{\text{tors}} \simeq
\begin{cases}
\Z / N\Z & 1\leq N \leq 10 \text{ or } N = 12,16,18,\\
\Z/2\Z \times \Z/2N\Z & 1\leq N \leq 6 \text{ or } N = 8,9.
\end{cases}
\end{equation*}

To complete the proof, it remains to eliminate $\Z/18\Z$ and $\Z/2\Z \times \Z/2N\Z$ for $N=7,9$. This can be done using \cite[Lemma 3.2]{Omer}, since $L$ has no cubic subfield and therefore has no subfield with degree divisible by $3$.
\end{proof}

\begin{remark}
We know that $\Z/2\Z\times \Z/10\Z$ is realized, as an example of an elliptic curve with this torsion subgroup over $\Q(\sqrt{-59})$ is given \cite[Section 6]{Omer}.  

From our \texttt{SageMath} calculations, we see that $\rk(E_{-23}(\Q)) = 1$. Then, by Proposition \ref{infinitely12}, we conclude that there exist infinitely many elliptic curves $E/\Q$ with torsion subgroup $\Z/2\Z \times \Z/12\Z$ over $\Q(\sqrt{-23})$.  

Both of these quadratic fields are contained in $L$, and we have shown that the torsion subgroup cannot be larger.  

We do not have explicit examples for the groups $\Z/16\Z$ or $\Z/2\Z \times \Z/16\Z$, but they cannot be eliminated and are likely to appear. Examples of $\Z/16\Z$ being realized over $\Q(\sqrt{d})$ are analyzed \cite{bander}, and an explicit example is given \cite{Omer}: the elliptic curve \lmfdbecx{266910}{ck}{6} satisfies
\begin{equation*}
E(\Q(\sqrt{41}))_{\text{tors}} \cong \Z/16\Z.
\end{equation*}

Similarly, the group $\Z/2\Z \times \Z/16\Z$ is known to be realized over $\Q(\sqrt{d_1}, \sqrt{d_2})$. An example is given \cite{Fujita}: the elliptic curve \lmfdbecx{210}{e}{7} satisfies
\begin{equation*}
E(\Q(\sqrt{-7}, \sqrt{-15}))_{\text{tors}} \cong \Z/2\Z \times \Z/16\Z.
\end{equation*}

Determining whether the specific numbers $d, d_1, d_2$ can be chosen to satisfy the conditions in the statement of this theorem is beyond the scope of this work.
\end{remark}

\begin{theorem}
Let $E/\Q$ be an elliptic curve.  
Define the field $L$ as
\begin{equation*}
    L = \Q\Big(\{\mu_{p^\infty} : p>11 \text{ is a prime satisfying } p \equiv 11 \Mod{12} \text{ and } p \not\equiv \pm 1 \Mod{5} \}\Big).
\end{equation*}
Then $ E(L)_{\text{tors}}$
is isomorphic to one of the groups from Mazur's theorem or one of the following groups:
\begin{equation*}
    \Z/2\Z \times \Z/12\Z, \quad \Z/16\Z, \quad \text{and} \quad \Z/2\Z \times \Z/16\Z.
\end{equation*}
\end{theorem}

\begin{proof}
  
The proof is almost the same as Theorem \ref{4exceptions}, but here we also need to eliminate $\Z/2\Z \times \Z/10\Z$.  

Suppose 
\begin{equation*}
    E(\Q(\zeta_n))_{\text{tors}} \cong \Z/2\Z \times \Z/10\Z
\end{equation*}
for some positive integer $n$ whose prime factors are all of the form $20k+3$ or $20k+7$.  
By Lemma \ref{upgradedfourpointpairing}, we can conclude that $    E_d(\Q)[5] \cong \Z/5\Z$
for some square-free integer $d$ with $\sqrt{d} \in \Q(\zeta_n)$.  

Since $E$ has full $2$-torsion defined over $\Q(\zeta_n)$, the quadratic twist $E_d$ also has full $2$-torsion over $\Q(\zeta_n)$.  
Because $\Q(\zeta_n)$ has no cubic subfield, the $2$-torsion of $E_d$ must be contained in a quadratic extension. Hence, there exists a quadratic field $K$ such that
\begin{equation*}
    E_d(K)_{\text{tors}} \cong \Z/2\Z \times \Z/10\Z
\end{equation*}
where $K$ is the field of definition of the $2$-torsion.  

The splitting field argument allows the possibility $K=\Q$, but this would contradict Mazur's theorem because no elliptic curve over $\Q$ has torsion $\Z/2\Z \times \Z/10\Z$.  
Thus $K$ must be a quadratic subfield of $\Q(\zeta_n)$, i.e., $K = \Q(\sqrt{D})$ for some square-free integer $D$. By Proposition \ref{multipleprimesfactored}, $D$ has a specific form, but Corollary \ref{noZ2Z10overmultiple} then implies that there is no elliptic curve over $\Q$ with 
\begin{equation*}
    E(\Q(\sqrt{D}))_{\text{tors}} \cong \Z/2\Z \times \Z/10\Z,
\end{equation*}
and therefore such a torsion subgroup cannot occur over $\Q(\zeta_n)$.
\end{proof}

\begin{remark}
   We can see that by reducing the density of the primes in our set from $3/16$ to $1/8$, we are able to eliminate one group structure. 
\end{remark}

\begin{remark}
We initially wanted to prove a similar theorem by replacing the condition 
$p \not\equiv 4 \pmod{5}$ with $p \not\equiv 7 \pmod{8}$, in order to show that 
$\Z/2\Z \times \Z/12\Z$ is not realized. This motivated the proof of Proposition \ref{twoprimes12}, which allows us to conclude that 
$\Z/2\Z \times \Z/12\Z$ is not realized over $\Q(\sqrt{d})$ for any 
$d \in \{\sqrt{-p},\sqrt{-pq}\}$ where $p,q$ are primes of the form $24k+11$.  

However, unlike the $\Z/2\Z \times \Z/10\Z$ case, the torsion $\Z/2\Z \times \Z/12\Z$ can still be realized over the biquadratic field 
$\Q(\sqrt{-p},\sqrt{-q})$ without being realized over any of its quadratic subfields.  

As an illustration (not a counterexample for our application, but giving the 
idea), consider the elliptic curve \lmfdbecx{30}{a}{6}. This curve has the desired 
torsion subgroup over three distinct biquadratic fields, yet its torsion over each 
quadratic subfield remains $\Z/2\Z \times \Z/6\Z$.
\end{remark}

\begin{theorem}\label{metabelian}
   Let $E/\Q$ be an elliptic curve, $p>7$ a prime, and $d$ an integer. Then
   \begin{equation*}
       E(\Q(\mu_{p^\infty},d^{1/p^\infty}))_{\text{tors}} = E(\Q(\zeta_p))_{\text{tors}}.
   \end{equation*}
\end{theorem}
\begin{proof}
The proof consists of two parts. In the first part, we show that $E$ cannot 
have full $p$-torsion over $\Q(\mu_{p^\infty},d^{1/p^\infty})$ by proving that if 
it did, then $E$ would have a rational $p$-torsion point, which contradicts 
Mazur's theorem. In the second part, we show that
\begin{equation*} 
E(\Q(\mu_{p^\infty},d^{1/p^\infty}))_{\mathrm{tors}} \cong \Z/N\Z 
\end{equation*}
or
\begin{equation*} 
E(\Q(\mu_{p^\infty},d^{1/p^\infty}))_{\mathrm{tors}} \cong \Z/2\Z \times \Z/2N\Z 
\end{equation*} 
for some positive integer $N$ such that
\begin{equation*}
    N\leq 19, \quad \text{or} \quad N\in\{21,25,27,37,43,67,163\}.
\end{equation*}
We then show that all of this torsion must be defined over $\Q(\zeta_p)$. 
Since $\Q(\mu_{p^\infty},d^{1/p^\infty})$ does not contain $\sqrt{-1}$, we know 
that $E$ cannot have full $4$-torsion over this field by the Weil pairing. Assume 
that 
\begin{equation*}
E(\Q(\mu_{p^\infty},d^{1/p^\infty}))[\ell] \cong \Z/\ell\Z \times \Z/\ell\Z
\end{equation*} 
for some odd prime $\ell$. Again, by the Weil pairing, our field must contain a 
primitive $\ell$-th root of unity, which implies that $\ell = p$. 

Next, we show that if $p>7$, then $E$ cannot have full $p$-torsion over this 
field. From Galois representations, we have the injection 
\begin{equation*} 
\Gal(\Q(E[p])/\Q) \hookrightarrow \GL_2(\Z/p\Z).
\end{equation*} 
Combining our assumption with the Weil pairing, we obtain
\begin{equation*} 
\Q(\zeta_p) \subset \Q(E[p]) \subset \Q(\mu_{p^\infty},d^{1/p^\infty}). 
\end{equation*} 
Hence, $\Z/(p-1)\Z$ is a quotient of $\Gal(\Q(E[p])/\Q)$, which is in turn a quotient of 
\begin{equation*}
\Gal(\Q(\mu_{p^\infty},d^{1/p^\infty})/\Q).
\end{equation*} 
This implies that the order of $\Gal(\Q(E[p])/\Q)$ is $(p-1)p^k$ for some $k \ge 0$. 
Since the order of $\GL_2(\Z/p\Z)$ is $p(p-1)^2(p+1)$, the order of 
$\Gal(\Q(E[p])/\Q)$ must be either $p-1$ or $p(p-1)$.

Let 
\begin{equation*} 
G = \rho(\Gal(\overline{\Q}/\Q)) \cong \Gal(\Q(E[p])/\Q)
\end{equation*} 
be the image of the Galois representation. We use the classification of maximal 
subgroups of $\GL_2(\Z/p\Z)$ due to Serre \cite{Serremaximal}, as well as \cite[Theorem 7.2]{maximalsubgroups}. 

If the order of $G$ is not divisible by $p$, then $G \cong \Z/(p-1)\Z$. Let $H$ be 
the projective image of $G$ in $\PGL_2(\Z/p\Z)$. Clearly, $H$ is cyclic, so $G$ is 
contained in a Cartan subgroup. We will show that $G$ must in fact be contained in 
a split Cartan subgroup. 

If the order of $G$ is divisible by $p$, then $G$ is contained in a Borel subgroup 
of $\GL_2(\Z/p\Z)$. In the first case, $G$ is cyclic of order $p-1$, hence abelian. 
Then, by \cite[Theorem 1.1]{abeliandivision}, we conclude that $p \geq 7$ is not 
possible. In the second case, by choosing a suitable basis for the $p$-torsion, we 
can write 
\begin{equation*} 
G \leq \Big\{ \begin{bmatrix} a & b\\ 0 & d \end{bmatrix} : a,b,d \in \Z/p\Z, \; a,d \not\equiv 0 \pmod{p} \Big\}.
\end{equation*} 

If the chosen basis for the $p$-torsion is given by points $P, Q \in E[p]$, then 
the action of $\Gal(\overline{\Q}/\Q)$ always sends $Q$ to a multiple of itself. 
Therefore, $\langle Q \rangle \subset E[p]$ is fixed by $\Gal(\overline{\Q}/\Q)$, 
and hence it is the kernel of a rational $p$-isogeny of $E$. This implies that $E$ 
must admit a rational $p$-isogeny, so by Theorem \ref{rationalisogeny} we obtain 
\begin{equation*}
p \le 19 \quad \text{or} \quad p \in \{37,43,67,163\}.
\end{equation*} 
If $p \in \{19,43,67,163\}$, then $E$ has complex multiplication by Theorem 
\ref{rationalisogeny}. Moreover, for each of these primes, there is a unique 
rational $j$-invariant corresponding to a rational $p$-isogeny \cite[Table 2]{Omer}.
Checking their geometric endomorphism rings on LMFDB, 
we see that they are 
\begin{equation*}
O_K = \Z\Big[\frac{1+\sqrt{-p}}{2}\Big], \quad \text{where} \quad K = \Q(\sqrt{-p}).
\end{equation*} 

From CM theory \cite[Corollary 5.7]{SilvermanA}, we know that if we adjoin all 
the $x$-coordinates of the torsion points of $E$ to $K$, we obtain the maximal 
abelian extension of $K$. Specifically, in our situation, the Galois group 
\begin{equation*}
\Gal(\Q(x(E[p]))/\Q(\sqrt{-p}))
\end{equation*} 
must be abelian. Observe that this extension has degree $p(p-1)/2$, which is odd. 

We also know that the extension 
\begin{equation*}
\Q(E[p])/\Q(x(E[p]))
\end{equation*} 
is Galois, with degree $1$, $2$, or $4$. It follows that 
\begin{equation*} 
\Q(E[p]) = \Q(x(E[p])),
\end{equation*} 
so that
\begin{equation*} 
\Gal(\Q(E[p])/\Q(\sqrt{-p})) 
\end{equation*} 
is abelian. However, by \cite[Theorem 1.1]{abeliandivision}, $G$ cannot be 
abelian if $p \ge 7$, so $G$ cannot be isomorphic to $C_{p(p-1)}$. This implies 
\begin{equation*}
\Q(E[p]) \cap \Q(\mu_{p^\infty}) = \Q(\zeta_p).
\end{equation*} 

Consequently, we must have 
\begin{equation*}
\Q(E[p]) = \Q(\zeta_p, d^{1/p})
\end{equation*} 
for some integer $d$, so that 
\begin{equation*}
G \cong C_p \rtimes C_{p-1}.
\end{equation*} 
Then
\begin{equation*}
\Gal(\Q(E[p])/\Q(\sqrt{-p})) \cong C_p \rtimes C_{(p-1)/2},
\end{equation*} 
which is non-abelian for any $p > 3$. Hence, the field of definition of the
$p$-torsion cannot be contained in $\Q(\mu_{p^\infty}, d^{1/p^\infty})$ if $p \in \{19,43,67,163\}$.

If $p=17$, there are two rational $j$-invariants corresponding to elliptic curves 
with a rational $17$-isogeny \cite{Omer}. The field of 
definition of such a rational $17$-isogeny is contained in $\Q(\zeta_{85})$, but 
not in $\Q(\zeta_5)$ or $\Q(\zeta_{17})$. Moreover, the degree of this extension 
divides $16$. Therefore, for any field of the form $\Q(\mu_{17^\infty}, 
d^{1/17^\infty})$, the intersection with the field of definition of the 
$17$-isogeny is contained in $\Q(\zeta_{17})$. This follows because the Sylow
$17$-subgroup of $\Gal(\Q(\mu_{17^\infty}, d^{1/17^\infty})/\Q)$ is unique,
so there is 
a unique subfield of degree $16$, namely $\Q(\zeta_{17})$. 

In conclusion, the full $17$-torsion cannot be contained in a field of the form 
$\Q(\mu_{17^\infty}, d^{1/17^\infty})$, since not even the kernel of the rational 
$17$-isogeny is contained there.  

A similar argument applies for $p=37$. There are two rational $j$-invariants 
corresponding to elliptic curves with a rational $37$-isogeny. For one of these 
curves, the field of definition of the $37$-isogeny is $\Q(\zeta_{35})^+$, and it 
is contained in $\Q(\zeta_{1295})$, but not in any $\Q(\zeta_n)$ where $n$ is a 
proper divisor of $1295$. Using the same reasoning as above, we again conclude 
that the full $37$-torsion cannot be contained in a field of the form 
$\Q(\mu_{37^\infty}, d^{1/37^\infty})$.  

Note that in \cite{Omer} we worked with the field of definition of the
$x$-coordinates of the torsion points. Hence, these fields do not become smaller even 
if we consider a quadratic twist of the same elliptic curves.

If $p=11$, the situation is slightly more complicated, since the field of 
definition of the kernel of the $11$-isogeny is contained in $\Q(\zeta_{11})$, so 
we cannot use the same arguments as in the $p=17$ or $p=37$ cases. Fortunately, 
there are only three different rational $j$-invariants, so we can use \texttt{SageMath} to 
compute the fields of definition of their $11$-torsion.  

The elliptic curve \lmfdbecx{121}{b}{2} has CM over $\Z\Big[\frac{1+\sqrt{-11}}{2}\Big]$, so we can apply the same argument as in the $p \in \{19,43,67,163\}$ 
case. For the other two curves, \texttt{SageMath} shows that the field of definition of the
$11$-torsion is of the form 
\begin{equation*}
\Q(\zeta_{11}, d^{1/11})
\end{equation*} 
for some $d \in \Q(\zeta_{11})$, but $d$ is not equivalent to a rational number up 
to $11$-th powers in $\Q(\zeta_{11})$. Therefore, this rules out the case $p=11$.

Finally, if $p=13$, the situation becomes more technical. We have shown that the 
order of $G$ is divisible by $13$, so $G$ is contained in a Borel subgroup and 
therefore $E$ has a rational $13$-isogeny. We know that there is no rational 
elliptic curve with CM and a $13$-rational isogeny simultaneously.  

If $E$ does not have CM, the possible images of the Galois representation are 
known by the work of Balakrishnan et al. \cite{Balakrishnan1, Balakrishnan2}. 
Another useful reference is the work of Rouse et al. \cite{PossibleGalois}, where 
they address the problem of determining $\ell$-adic images of Galois for rational 
elliptic curves for all primes $\ell$. Since $G$ cannot be abelian 
\cite{abeliandivision}, it would have to be isomorphic to 
\begin{equation*}
C_{13} \rtimes C_{12},
\end{equation*} 
but this is not possible by the classification of $13$-adic images of Galois.  

Hence, we conclude that the torsion subgroup 
\begin{equation*}
E(\Q(\mu_{p^\infty}, d^{1/p^\infty}))_{\mathrm{tors}}
\end{equation*} 
must be one of the forms described previously.

In the rest of the proof, we first show that all $2^\infty$-torsion belongs to $\Q(\zeta_p)$, and then we show that all odd-order torsion belongs to $\Q(\zeta_p)$.  

Let us show that 
\begin{equation*}
E(\Q(\mu_{p^\infty}, d^{1/p^\infty}))[2^\infty] = E(\Q(\zeta_p))[2^\infty].
\end{equation*} 
Since our field does not contain $\sqrt{-1}$, we have 
\begin{equation*} 
E(\Q(\mu_{p^\infty}, d^{1/p^\infty}))[2^\infty] \cong \Z/2^k\Z \quad \text{or} \quad 
E(\Q(\mu_{p^\infty}, d^{1/p^\infty}))[2^\infty] \cong \Z/2\Z \times \Z/2^{k+1}\Z
\end{equation*} 
for some $k \ge 0$. By Lemma \ref{n-isogeny}, $E$ has a rational $2^k$-isogeny, so $k \le 4$ by Theorem \ref{rationalisogeny}. Therefore, 
\begin{equation*} 
E(\Q(\mu_{p^\infty}, d^{1/p^\infty}))[2^\infty] \subseteq E[32].
\end{equation*} 

Let 
\begin{equation*} 
L = \Q(E[32]) \cap \Q(\mu_{p^\infty}, d^{1/p^\infty}).
\end{equation*} 
We will show that the order of $\Gal(L/\Q)$ is not divisible by $p$. This implies 
that $\Gal(L/\Q)$ has order dividing $p-1$. Moreover, $\Gal(L/\Q)$ is a quotient of 
\begin{equation*} 
\Gal(\Q(\mu_{p^\infty}, d^{1/p^\infty})/\Q),
\end{equation*} 
so its kernel must contain the unique Sylow $p$-subgroup. Hence, $L \subset \Q(\zeta_p)$.  

To show that $\Gal(L/\Q)$ is not divisible by $p$, it suffices to show that 
\begin{equation*} 
\Gal(\Q(E[32])/\Q)
\end{equation*} 
has order not divisible by $p$, since $\Gal(L/\Q)$ is a quotient of this group. 
We know that 
\begin{equation*} 
\Gal(\Q(E[32])/\Q) = 3 \cdot 2^{17},
\end{equation*} 
so $p$ does not divide this order for all $p>3$. Therefore, all $2$-primary torsion over $\Q(\mu_{p^\infty}, d^{1/p^\infty})$ is defined over $\Q(\zeta_p)$.  

Now let $q$ be an odd prime power dividing $N$, where $N$ is the integer appearing 
in the description of $E(\Q(\mu_{p^\infty}, d^{1/p^\infty}))_{\mathrm{tors}}.$
We will show that 
\begin{equation*} 
E(\Q(\mu_{p^\infty}, d^{1/p^\infty}))[q] = E(\Q(\zeta_p))[q].
\end{equation*} 
Observe that $E$ has a rational $N$-isogeny by Lemma \ref{n-isogeny}, so 
\begin{equation*} 
q \in \{1,3,5,7,11,13,17,19,25,27,37,43,67,163\}.
\end{equation*} 
If $q=67$, the field of definition of the kernel of the rational $67$-isogeny 
contains $\Q(\zeta_{67})^+$ \cite{Omer}, so it is not contained in a 
field of the form $\Q(\mu_{11^\infty}, d^{1/11^\infty})$. In all other cases, $p$ 
does not divide $\phi(q)$.  

Since $E(\Q(\mu_{p^\infty}, d^{1/p^\infty}))_{\mathrm{tors}}$ is finite, let $L$ 
be the field of definition of this torsion subgroup. Then $L/\Q$ is a Galois 
extension. Moreover, $\Gal(L/\Q(\zeta_p))$ is the unique Sylow $p$-subgroup of 
$\Gal(L/\Q)$, so it is normal. Hence, $L/\Q(\zeta_p)$ is Galois.  

Finally, the result follows immediately from Lemma \ref{upgradedfourpointpairing}, 
taking $K = \Q(\zeta_p)$, since $\Gal(L/\Q(\zeta_p))$ is a $p$-group and the order 
of any element is a power of $p$. Since $p$ does not divide $\phi(q)$, we can 
apply the first part of the lemma to conclude.
\end{proof}

\begin{remark}
  The problem of whether $E$ can have full $p$-torsion over the field $\Q(\mu_{p^\infty}, d^{1/p^\infty})$
is an interesting one. We have shown that this is not possible when $p > 7$. There are examples of elliptic curves with full $p$-torsion over $\Q(\zeta_p)$ for $p = 2, 3, 5$, which settles the problem for these primes.  

The case $p = 7$ is particularly interesting. Full $7$-torsion cannot be realized 
over $\Q(\zeta_7)$ \cite{abeliandivision}, but in principle it could be realized over
a field of the form $\Q(\mu_{7^\infty}, d^{1/7^\infty}).$
Indeed, if we allow $d$ to lie in $\Q(\zeta_7)$ rather than in $\Q$, the elliptic 
curve \lmfdbecx{26}{b}{2} has full $7$-torsion over a Kummer extension $\Q(\zeta_7, d^{1/7})$
for some $d \in \Q(\zeta_7)$, as confirmed by our \texttt{SageMath} calculations. The extension has the same group-theoretical properties whether $d$ is rational or not: it is a degree $42$ extension with Galois group isomorphic to $C_7 \rtimes C_6.$

However, the question of whether there exists an elliptic curve $E/\Q$ with full $7$-torsion over $\Q(\zeta_7, d^{1/7})$
for some \emph{rational} $d$ remains open and is beyond the scope of our work. This problem is considerably harder than the $p=11$ case, because there are only three different rational $j$-invariants with a rational $11$-isogeny, whereas there are infinitely many different rational $j$-invariants with a rational $7$-isogeny. Therefore, we cannot check all of them manually as we did in the proof of the $p=11$ case.
\end{remark}

Another interesting application of Lemma \ref{upgradedfourpointpairing} concerns $\Z_p$-extensions. Let $K$ be a number field, and let $K_{\Z_p}$ denote the composite of all $\Z_p$-extensions of $K$.  

We know that $K_{\Z_p}/K$ is abelian, and 
\begin{equation*}
\Gal(K_{\Z_p}/K) \cong \Z_p^a
\end{equation*} 
for some nonnegative integer $a = a_p(K)$ satisfying
\begin{equation*}
r_2 + 1 \le a_p(K) \le d,
\end{equation*} 
where $r_2$ denotes the number of pairs of complex conjugate embeddings of $K$, and $d = [K:\Q]$ is the degree of $K$ over $\Q$ \cite[Theorem 2-3]{iwasawa}.  

This allows us to state torsion classification theorems in terms of the composite of all $\Z_p$-extensions of $K$, providing a natural setting to analyze $p$-primary torsion over infinite abelian extensions.

\begin{theorem}\label{Zp-extensions}
Let $E/\Q$ be an elliptic curve. Let $K$ be a number field, and let $K_{\Z_p}$ denote the composite of all $\Z_p$-extensions of $K$ for some prime $p>5$. Let $q$ be an odd prime power such that
\begin{equation*}
E(K_{\Z_p})[q] \cong \Z/q\Z.
\end{equation*}
Then
\begin{equation*}
E(K_{\Z_p})[q] = E(K)[q].
\end{equation*}
\end{theorem}

\begin{proof}
We know that $K_{\Z_p}/\Q$ is a Galois extension. By Lemma \ref{n-isogeny}, it follows that $E$ has a rational $q$-isogeny. From Theorem \ref{rationalisogeny}, we know all possible values of $q$, and any prime factor of $\phi(q)$ is at most $11$. Assuming $p>11$, we ensure that $p$ does not divide $\phi(q)$.  

Let $L$ denote the smallest Galois extension of $\Q$ containing both $K$ and $\Q(E(K_{\Z_p})[q])$. Then $L/K$ is a finite Galois extension, and $\Gal(L/K)$ is a $p$-group. Applying Lemma \ref{upgradedfourpointpairing}, we conclude that
\begin{equation*}
E(K_{\Z_p})[q] = E(K)[q].
\end{equation*}

In the exceptional cases $(p,q) = (7,43)$ or $(11,67)$, the theorem does not apply directly. 
However, the result still holds, since $K_{\Z_p}$ does not contain the field $\Q(\zeta_q)^+$ 
\cite[Theorem 1.1, proof]{Omer2}, \cite[Section 6]{Omer}.
\end{proof}

As a natural consequence of Theorem \ref{Zp-extensions}, we can generalize our 
results in \cite{Omer2}. Let $K$ be a quadratic field. If $K$ is real quadratic, then 
$K_{\mathrm{cyc}}$ is the unique $\Z_p$-extension of $K$, which is the case we have 
already treated.  

Now, let $K$ be an imaginary quadratic field. In this case, there are infinitely many distinct $\Z_p$-extensions of $K$, all contained in the compositum
\begin{equation*}
K_{\Z_p} = K_{\mathrm{cyc}} K_{\mathrm{anti}},
\end{equation*}
satisfying
\begin{equation*}
\Gal(K_{\Z_p}/K) \cong \Z_p^2.
\end{equation*}

Instead of considering a single $\Z_p$-extension of $K$, we may work with the entire 
field $K_{\Z_p}$ and classify torsion over it. The following result can be proven in 
a manner similar to \cite[proof of Theorem 1.1]{Omer2}, but it also follows 
directly as a consequence of Theorem \ref{Zp-extensions}.

\begin{theorem}\label{Z_p^2}
Let $E/\Q$ be an elliptic curve. Let $K$ be an imaginary quadratic field, and let $p>5$ be a prime. Let 
\begin{equation*}
K_{\Z_p} = K_{\mathrm{cyc}} K_{\mathrm{anti}}
\end{equation*} 
denote the extension satisfying 
\begin{equation*}
\Gal(K_{\Z_p}/K) \cong \Z_p^2.
\end{equation*} 
Then
\begin{equation*}
E(K_{\Z_p})_{\mathrm{tors}} = E(K)_{\mathrm{tors}}.
\end{equation*}
\end{theorem}

\begin{proof}
Let $n$ be the number of roots of unity in $K_{\Z_p}$. By a similar argument to \cite[Lemma 4.7]{Omer2}, we have
\begin{equation*}
n = 
\begin{cases}
6 & \text{if } K = \Q(\sqrt{-3}),\\
4 & \text{if } K = \Q(\sqrt{-1}),\\
2 & \text{otherwise}.
\end{cases}
\end{equation*}

Torsion subgroup $E(K_{\Z_p})_{\mathrm{tors}}$ is finite and satisfies
\begin{equation*}
E(K_{\Z_p})_{\mathrm{tors}} \cong \Z/m\Z \times \Z/mN\Z,
\end{equation*}
where $m \in \{1,2,3,4,6\}$ and $N$ is one of the integers appearing in Theorem \ref{rationalisogeny} \cite[Lemma 4.1]{Omer2}.  

All $2^\infty$-torsion and $3^\infty$-torsion is defined over $K$ \cite[Lemma 4.3, 4.4, 4.5]{Omer2}. By Theorem \ref{Zp-extensions}, all $q$-torsion is defined over $K$ for any $q$ that is a power of a prime greater than $3$ and divides $N$.  

Therefore, all torsion over $K_{\Z_p}$ is already defined over $K$, which proves the theorem.
\end{proof}

Finally, we present a corollary that combines our results in \cite{Omer2}, Theorem \ref{Z_p^2}, and the results in Section \ref{descentsection}.

\begin{corollary}
Let $K = \Q(\sqrt{d})$ be a quadratic field, and let $p>5$ be a prime. Let $K_{\Z_p}$ denote the composite of all $\Z_p$-extensions of $K$. Then:
\begin{enumerate}[(i)]
    \item If $d$ is a product of $-p_i$ for $i=1,\dots,n$, where $p_i$ are 
    distinct primes satisfying $p_i \equiv 3,7 \pmod{20}$, then there is no 
    elliptic curve $E/\Q$ with $E(K_{\Z_p})_{\mathrm{tors}} \cong \Z/2\Z \times \Z/10\Z.$

    \item If $d$ is of the form $-p_1$ or $p_1 p_2$, where $p_1, p_2$ are distinct
    primes satisfying $p_1, p_2 \equiv 11 \pmod{24}$, then there is no 
    elliptic curve $E/\Q$ with $  E(K_{\Z_p})_{\mathrm{tors}} \cong \Z/2\Z \times \Z/12\Z.$

    \item If $d = -p$ for any prime number $p$, then there is no elliptic curve $E/\Q$ with $  E(K_{\Z_p})_{\mathrm{tors}} \cong \Z/16\Z.$
\end{enumerate}
\end{corollary}

\begin{proof}
From \cite[Theorem 1.1]{Omer2} and Theorem \ref{Z_p^2}, we have
\begin{equation*}
E(K_{\Z_p})_{\mathrm{tors}} = E(K)_{\mathrm{tors}}.
\end{equation*}
The conclusion then follows from Corollaries \ref{noZ2Z10overmultiple}, \ref{prime12}, \ref{twoprimes12}, and \cite[Theorem 5.2]{derickx16}.
\end{proof}

\bibliography{references}

\end{document}